\documentclass{article}

     \usepackage[preprint,nonatbib]{neurips_2020}

\usepackage{enumitem}
\usepackage[utf8]{inputenc} %
\usepackage[T1]{fontenc}    %
\usepackage{hyperref}       %
\usepackage{url}            %
\usepackage{booktabs}       %
\usepackage{amsfonts}       %
\usepackage{nicefrac}       %
\usepackage{microtype}      %
\usepackage{amsmath, amssymb, amsfonts, amsthm}
\usepackage{comment}
\usepackage[square,numbers]{natbib}
\newcommand{\vect}[1]{\ensuremath{\mathbf{#1}}}

\newcommand{\grad}{\nabla}

\newcommand{\argmin}{\mathop{\mathrm{argmin}}}

\newcommand{\iprod}[2]{\langle #1, #2 \rangle}

\newcommand{\abs}[1]{\left|{#1}\right|}
\newcommand{\norm}[1]{\|{#1} \|}

\newcommand{\poly}{\mathrm{poly}}

\newcommand{\cP}{\mathcal{P}}

\newcommand{\R}{\mathbb{R}}

\newcommand{\E}{\mathbb{E}}

\renewcommand{\u}{\vect{u}}

\newcommand{\cX}{\mathcal{X}}
\newcommand{\cY}{\mathcal{Y}}

\newcommand{\order}[1]{O\left(#1\right)}

\usepackage[linesnumbered,ruled,vlined,algonl]{algorithm2e}
\SetKwInput{Init}{Initialize}
\SetKwComment{Comment}{$\triangleright$\ }{}
\makeatletter
\newcommand{\nosemic}{\renewcommand{\@endalgocfline}{\relax}}%
\newcommand{\dosemic}{\renewcommand{\@endalgocfline}{\algocf@endline}}%
\let\oldnl\nl%
\newcommand{\nonl}{\renewcommand{\nl}{\let\nl\oldnl}}%
\makeatother

\usepackage{color}
\usepackage{xcolor}
\definecolor{skyblue}{rgb}{0.529,0.808,0.922}
\definecolor{LightCyan}{rgb}{0.88,1,1}
\definecolor{LightLightCyan}{rgb}{1.0,1.0,1.0}
\definecolor{Gray}{gray}{0.85}
\definecolor{DarkGray}{gray}{0.95}
\usepackage{colortbl}

\usepackage{mathtools}

\newcommand{\defeq}{\mathrel{\mathop:}=}
\newcommand{\prox}{\textrm{prox}}
\newcommand{\sfo}[1]{\text{SFO}\left(#1\right)}
\newcommand{\fo}[1]{\text{FO}\left(#1\right)}
\newcommand{\po}[1]{\text{PO}\left(#1\right)}
\newcommand{\lmo}[1]{\text{LMO}\left(#1\right)}

\def\ones{\mathbf{1}}

\def\reals{\mathbb{R}}

\def\cX{\mathcal{X}}
\def\bX{\mathcal{X}'}
\def\vX{{\reals^d}}

\def\hg{\widehat{g}}
\def\tg{\widetilde{g}}
\def\tu{\widetilde{u}}
\def\hT{\widehat{T}}

\def\indctr{{\mathbb I}}

\def\hx{\widehat{x}}

\def\prox{{\rm prox}}

\newcommand{\Ip}[2]{\left\langle#1, #2\right\rangle}

\def\Ord{{\mathcal{O}}}
\def\ord{{{o}}}

\renewcommand{\order}[1]{{\mathcal O}\left(#1\right)}

\def\px{x}
\def\mx{y}
\def\dx{z}
\def\tdx{\widetilde{z}}

\DeclarePairedDelimiter{\ceil}{\lceil}{\rceil}

\newtheorem{theorem}{Theorem}
\newtheorem{lemma}{Lemma}
\newtheorem{proposition}{Proposition}

\newtheorem{definition}{Definition}

\newenvironment{proofsketch}{%
	\proof}{\endproof}

\newcommand{\projstep}{\texttt{\selectfont Approx-Proj}\xspace}
\newcommand{\proxslide}{\texttt{\selectfont Prox-Slide}\xspace}

\newcommand{\aproj}{MOPES\xspace}%
\newcommand{\almo}{MOLES\xspace}%
\newcommand{\psgd}{PGD\xspace}
\newcommand{\apgd}{APGD\xspace}
\newcommand{\iapgd}{IAPGD\xspace}

\newcommand{\fwpsgd}{FW-PGD\xspace}
\newcommand{\randfw}{RandFW\xspace}

\newcommand{\focc}{FO-CC\xspace}
\newcommand{\sfocc}{S\focc}
\newcommand{\pocc}{PO-CC\xspace}
\newcommand{\lmocc}{LMO-CC\xspace}
\newcommand{\nsco}{NSCO\xspace}
\newcommand{\fl}{f_\lambda}
\usepackage{empheq}
\newcommand*\widefbox[1]{\fbox{\hspace{2em}#1\hspace{2em}}}

\usepackage{cleveref}

\newcommand{\papertitle}{Projection Efficient Subgradient Method and Optimal Nonsmooth Frank-Wolfe Method}
\title{\papertitle}

\author{%
	Kiran Koshy Thekumparampil \\
	University of Illinois at Urbana-Champaign\\
	\texttt{thekump2@illinois.edu} \\
	\And
	Prateek Jain\\
	Microsoft Research, India\\
	\texttt{prajain@microsoft.com} \\
	\And
	Praneeth Netrapalli\\
	Microsoft Research, India\\
	\texttt{praneeth@microsoft.com}\\
	\And
	Sewoong Oh\\
	University of Washington, Seattle\\
	\texttt{sewoong@cs.washington.edu}\\
}

\begin{document}

\maketitle

\begin{abstract}
We consider the classical setting of optimizing a nonsmooth Lipschitz continuous convex function over a convex constraint set, when having access to a (stochastic) first-order oracle (FO) for the function and a projection oracle (PO) for the constraint set. 
It is well known that 
to achieve $\varepsilon$-suboptimality in high-dimensions, $\Theta(\varepsilon^{-2})$ FO calls are necessary~\cite{nemirovsky1983problem}. This is achieved by the projected subgradient method (\psgd)~\cite{bertsekas1999nonlinear}.
However, \psgd~also entails $\Ord(\varepsilon^{-2})$ PO calls, which may be computationally costlier than FO calls (e.g.~nuclear norm constraints). 
Improving this PO calls complexity of \psgd~is largely unexplored, 
despite the fundamental nature of this problem and extensive literature. 
We present first such improvement. %
This only requires a mild assumption that the objective function, when extended to a slightly larger neighborhood of the constraint set, still remains Lipschitz and accessible via FO. 
In particular, we introduce \aproj~method, which carefully combines Moreau-Yosida smoothing and accelerated first-order schemes. 
This is guaranteed to find a \emph{feasible} $\varepsilon$-suboptimal solution using only $\Ord(\varepsilon^{-1})$ PO calls and optimal $\Ord(\varepsilon^{-2})$ FO calls. 
Further,
{instead of a PO if we only have a linear minimization oracle (LMO, \`a la Frank-Wolfe) to access the constraint set},
an extension of our method, \almo, finds a \emph{feasible} $\varepsilon$-suboptimal solution using $\Ord(\varepsilon^{-2})$ LMO calls and FO calls---both match known lower bounds~\cite{lan2013complexity},  
resolving a question left open since~\cite{white1993extension}. 
{Our experiments confirm that these methods achieve significant speedups over the state-of-the-art, 
for a problem with costly PO and LMO calls.
}

\end{abstract}

\section{Introduction}
\label{sec:intro}

In this paper, we consider the nonsmooth convex optimization (\nsco) problem with the First-order Oracle (FO) and the Projection Oracle (PO) defined as: 
\begin{equation}\label{eq:nsco}
\text{NSCO}:\min_x\ f(x),\ \text{s.t.}\ x \in \cX\ , \ \  \text{FO}(x)\in \partial f(x),\ \text{and } \text{PO}(x)= \mathcal{P}_\cX(x) = \argmin_{y\in \cX}\|y-x\|_2^2,
\end{equation}
where  $f: \R^d \rightarrow \R$ is a convex Lipschitz-continuous function, and $\cX \subseteq \R^d$ is a convex constraint. When queried at a point $x$, FO returns a subgradient of $f$ at $x$ and PO returns the projection of $x$ onto $\cX$. \nsco is a fundamental problem with a long history and several important applications including support vector machines (SVM) \cite{bishop2006pattern}, robust learning \cite{huber1996robust}, and utility maximization in finance \cite{Vinter-Zheng}.

Finding an $\varepsilon$-suboptimal solution for this problem requires $\Omega(\varepsilon^{-2})$ FO calls in the worst case, when the dimension $d$ is large~\cite{nemirovsky1983problem}. This lower 
bound is %
tightly 
matched by the projected subgradient method (PGD).
Unfortunately, PGD also uses 
one PO call after every FO call,
resulting in a PO calls complexity (PO-CC)---the number of times PO needs to be invoked---of ${\Theta}(\varepsilon^{-2})$. 
This can be a major bottleneck in solving 
several practical problems like collaborative filtering \cite{srebro2005maximum}, %
where the cost of a PO is often higher than the cost of an FO call. This begs the natural question, which surprisingly is largely unexplored in the general nonsmooth optimization setting: {\em Can we design an algorithm whose PO calls complexity is significantly
better than the optimal FO calls complexity $O(\varepsilon^{-2})$?}

\begin{table}[t]
	\begin{center}
		\label{tab:results}
		{\resizebox{\textwidth}{!}{\small
				\begin{tabular}{c c  c c c}
					\toprule
					& {\bf Randomized Smoothing} &{\bf State-of-the-art} & {\bf Our results} & {\bf Lower}\\
					& {\bf dimension dependent} & {\bf dimension-free} & { (Theorems~\ref{thm:envelope-subgrad-method-cor1} and\ref{thm:envelope-subgrad-method-cor2}) } & {\bf bound}\\	\midrule 
					\rowcolor{LightLightCyan}
					SFO  & $\Ord((G^2+\sigma^2)/\varepsilon^2)$~\cite{duchi2012randomized} & $\Ord((G^2+\sigma^2)/\varepsilon^2)$~\cite{nesterov1998introductory} &  \cellcolor{LightCyan} {\color{black} $\boldsymbol{\Ord((G^2+\sigma^2)/\varepsilon^2)}$} & $\Omega((G^2+\sigma^2)/\varepsilon^2)$~\cite{nemirovsky1983problem} \\ 
					\rowcolor{DarkGray}
					PO &  $\Ord(d^{1/4}G/\varepsilon)$~\cite{duchi2012randomized} & $\Ord(G^2/\varepsilon^2)^\star$~\cite{nesterov1998introductory} & \cellcolor{Gray} {\color{black} $\boldsymbol{\Ord(G/\varepsilon)}$} & Open problem \\ \midrule 
					\rowcolor{LightLightCyan}
					SFO  &  $\Ord(\sqrt{d}\,(G^2+\sigma^2)^2/\varepsilon^4)$~\cite{lan2013complexity} &  $\Ord((G^2+\sigma^2)/\varepsilon^2)^\dagger$ & \cellcolor{LightCyan} {\color{black} $\boldsymbol{\Ord((G^2+\sigma^2)/\varepsilon^2)}$} & $\Omega((G^2+\sigma^2)/\varepsilon^2)$~\cite{nemirovsky1983problem} \\ 
					\rowcolor{DarkGray}
					LMO &  $\Ord(\sqrt{d}\,G^2/\varepsilon^2)$~\cite{lan2013complexity} & $\Ord((G^2+\sigma^2)^2/\varepsilon^4)^\dagger$ &  \cellcolor{Gray} {\color{black} $\boldsymbol{\Ord(G^2/\varepsilon^2)}$}& $\Omega(G^2/\varepsilon^2)$~\cite{lan2013complexity}  \\ 
					\bottomrule 
				\end{tabular}
		}}
	\end{center}
	\caption{
		{Comparison of SFO~\eqref{eq:sfo}, PO~\eqref{eq:nsco} \& LMO~\eqref{eq:lmo} calls complexities of our methods and state-of-the-art algorithms, and corresponding lower-bounds for finding an approximate minimizer of a $d$-dimensional NSCO problem~\eqref{eq:nsco}.}
		We assume that $f$ is convex and $G$-Lipschitz continuous, and is accessed through a stochastic subgradient oracle with a variance of $\sigma^2$.
		$^\star$requires using a minibatch of appropriate size, 
		$^\dagger$approximates projections of \psgd~with FW method (\fwpsgd, see Appendix~\ref{sec:fw_proj_subgrad_method}).}
\end{table}

In this work, we answer the above question in the affirmative. Our first key contribution is MOreau Projection Efficient Subgradient method (\aproj), that obtains an $\varepsilon$-suboptimal solution using only $\Ord(\varepsilon^{-1})$ PO calls, while still ensuring that the FO calls complexity (FO-CC)---the number of times FO needs to be invoked---is optimal, i.e., $\Ord(\varepsilon^{-2})$. This requires a mild assumption that the function $f$ extends to a slightly larger neighborhood of the constraint set $\cX$. Concretely, we assume that $f$ is Lipschitz continuous in this neighborhood and FO can be queried at points in this neighborhood. To the best of our knowledge, our result is the first improvement over the $\Ord({\varepsilon^{-2}})$ PO calls of PGD for minimizing a general nonsmooth Lipschitz continuous convex function.

We achieve this by carefully combining Moreau-Yosida regularization with accelerated first-order methods \cite{moreau1962functions,tseng2008accelerated}. 
As accelerated  methods cannot be directly applied to a nonsmooth $f$, we can instead apply them to 
minimize its Moreau envelope, which is smooth (as long as $f$ is Lipschitz continuous). 
Although this idea has been explored, for example, in~\cite{devolder2014first,beck2012smoothing},  
\pocc
has remained $\Ord({\varepsilon^{-2}})$,  unless a much stronger and unrealistic oracle is assumed~\cite{beck2012smoothing} with a direct access to the  gradient of Moreau envelope. 
The key idea in breaking this barrier is 
to separate out the dependence on FO calls of $f$ from PO calls to $\cX$ by:  
($a$) using 
 Moreau-Yosida regularization to \emph{split} 
 the original problem into a composite problem, where one component consists of an unconstrained optimization of the function $f$ and the other consists of a simple constrained optimization over the set $\cX$; and 
($b$) applying the gradient sliding algorithm~\cite{lan2016gradient} on this joint problem to ensure the above mentioned bounds for both FO and PO calls. 
{We note that our results are limited to the Euclidean norm, since our results crucially depend on smoothness of the Moreau envelope and its regularizer, which is not known for Moreau envelopes based on general Bregman divergences~\cite{bauschke2018regularizing}.} 

In some high-dimensional problems, 
even a single call to the PO can be computationally prohibitive. 
A popular alternative, pioneered by 
\citet{frank1956algorithm},
is to replace PO by a more efficient 
Linear Minimization Oracle (LMO), 
which returns a minimizer of 
any linear functional $\Ip{g}{\cdot}$ over the set $\cX$.
\begin{align}\label{eq:lmo}
\lmo{g} \in \argmin_{s\in \cX} \Ip{g}{s}
\end{align}
Linear minimization is much faster than projection 
in several practical ML applications such as
a nuclear norm ball constrained problems \cite{cai2010singular},
video-narration alignment~\cite{alayrac2016unsupervised}, structured SVM~\cite{lacoste2013block}, and multiple sequence alignment and motif discovery~\cite{yen2016convex}. 
LMO based methods have
an important additional benefit of producing solutions that preserve desired structures such as sparsity and low rank.
For {\em smooth} $f$, 
there is a long history of conditional gradient (Frank-Wolfe) methods that use 
$\Ord(\varepsilon^{-1})$ LMO calls and $\Ord(\varepsilon^{-1})$ FO calls to achieve $\varepsilon$-suboptimality, which achieve optimal \lmocc~\cite{jaggi2013revisiting}. 
For {\em nonsmooth} functions, starting from the work of~\cite{white1993extension}, several approaches have been proposed,
some under more assumptions. 
The best known upper bound on LMO calls is $\Ord(\sqrt{d}\varepsilon^{-2})$ which is achieved at the expense of significantly larger $\Ord(\varepsilon^{-4})$ FO calls. %
Details of these are in Section~\ref{sec:related}.

Our second key contribution is the algorithm~\almo,~which obtains an $\varepsilon$-suboptimal solution using the optimal $\Ord({\varepsilon^{-2}})$ LMO and FO calls, without any additional dimension dependence. 
We achieve this result by extending~\aproj~to work with approximate projections and using the classical Frank-Wolfe (FW) method~\cite{frank1956algorithm} to implement these approximate projections using LMO calls.

Finally, both of our methods extend naturally to the Stochastic First-order Oracle (SFO) setting, where we have access only to stochastic versions of the function's subgradients. Stochastic versions of \aproj~and \almo~still achieve the  
the same PO/LMO calls complexities as deterministic counterparts, while the SFO calls complexity (\sfocc) is $\order{(1+\sigma^2)\epsilon^{-2}}$, where $\sigma^2$ is the variance in SFO. This again matches information theoretic lower bounds~\cite{nemirovsky1983problem}.

\textbf{Contributions}: We  summarize our contributions below and in Table~\ref{tab:results}. We assume that the function $f$ extends to a slightly larger neighborhood of the constraint set $\cX$ i.e., $f$ continues to be Lipschitz continuous and (S)FO can be queried in this neighborhood.
\begin{itemize}[leftmargin=7pt]
	\item We introduce~\aproj~and show that it is guaranteed to find an $\varepsilon$-suboptimal solution for any constrained nonsmooth convex optimization problem using $\Ord({\varepsilon^{-1}})$ PO calls and optimal $\Ord({\varepsilon^{-2}})$ {SFO} calls. To the best of our knowledge, for the general problem, this achieves  the first improvement over $\Ord({\varepsilon^{-2}})$ 
	\pocc and \sfocc 
	of stochastic projected subgradient method (PGD).
	\item For LMO setting, we extend our method to design \almo, that achieves the optimal 
	\sfocc and \lmocc
	of $\Ord({\varepsilon^{-2}})$, and improves over the  best known 
	\lmocc 
	by $\sqrt{d}$.
	\item We also empirically evaluate \aproj and \almo on the popular 
	nuclear norm constrained Matrix SVM 
	problem \cite{wolf2007modeling}, where they achieve significant speedups over their corresponding baselines. %
	\item Our main technical novelty is the use Moreau-Yosida regularization to separate out the constraint (PO/LMO) and function (SFO) accesses into two parts of a composite optimization problem. This enables a better control of how many times each of these oracles are accessed. This idea might be of independent interest, whenever a trade-off between \pocc/\lmocc
	and 
	\sfocc
	is desirable.
\end{itemize}

\subsection{Related Work}
\label{sec:related}
{\bf Nonsmooth convex optimization}: Nonsmooth convex optimization has been the focal point of several research works for past few decades. \cite{nemirovsky1983problem} provided information theoretic lower bound of FO calls $O(\varepsilon^{-2})$ to obtain $\varepsilon$-suboptimal solution, for the general problem. This bound is matched by the \psgd~method introduced independently by~\citep{goldstein1964convex} and~\citep{levitin1966constrained}, which also implies a
\pocc
of $O(\varepsilon^{-2})$. Recently, several faster \psgd~style methods \cite{lacoste2012simpler,shamir2013stochastic,yang2018rsg,kundu2018convex} have been proposed that exploit more structure in the given optimization function, e.g., when the function is a sum of a smooth and a nonsmooth function for which a {\em proximal} operator is available \cite{beck2009fast}. But, to the best of our knowledge, such works do not explicitly address \pocc and are mainly concerned about optimizing 
\focc.
Thus, for the worst case nonsmooth functions, these methods still suffer from $O(\varepsilon^{-2})$ \pocc. 

{\bf Smoothed surrogates}: Smoothing of the nonsmooth function is another common approach in solving them \cite{moreau1962functions,nesterov2005smooth}. In particular, randomized smoothing \cite{duchi2012randomized,beck2012smoothing} techniques have been successful in bringing down 
\focc %
w.r.t. $\varepsilon$ but such improvements come at the cost of dimension factors. For example, \cite[Corollary 2.4]{duchi2012randomized} provides  a randomized smoothing method that has 
$O(d^{1/4}/\varepsilon)$ \pocc and $O(\varepsilon^{-2})$ \focc.
Our \aproj~method guarantees significantly better \pocc~than \psgd~that is still {\em independent} of dimension. 

{\bf One or $\log({1}/{\epsilon})$ projection methods}: Starting with the work of \cite{mahdavi2012stochastic}, several recent works \cite{zhang2013logt,chen2013optimal,yang2017richer} have proposed methods that require only {\em one} or  $\log({1}/{\epsilon})$ projections, under a variety of conditions on the optimization function like smoothness and strong convexity. However, these methods require that the constraint set can be written as $c(x)\leq 0$ and they require access to $\grad c(x)$---the gradient of $c$--in {\em each} iteration. Hence, for the general nonsmooth functions, they will require at least $O(\varepsilon^{-2})$ accesses to gradients of the set's functional form. On the other hand, our method is required to access the set at only $O(\varepsilon^{-1})$ points. Furthermore, for several practical problems, the computational complexities of computing $\grad c(x)$ and projecting are similar. For example, when $c(x)=\|x\|_{\rm nuc} - r$ where $\|\cdot \|_{\rm nuc}$ denotes the nuclear norm (see Section~\ref{sec:applications}), then both gradient of $c(x)$ as well as PO requires computation of a full-SVD of $x$. 

{\bf Frank-Wolfe methods:} FW or {\it conditional gradient} method~\cite{frank1956algorithm,levitin1966constrained}
for smooth convex optimization, which uses LMO, has found renewed interest in machine learning \cite{zhang2003sequential,jaggi2013revisiting} due to the efficiency of computing LMO over PO \cite{gidel2018frank}, and its ability to ensure atomic structure and provide coreset guarantees~\cite{clarkson2010coresets}. Over the last decade, several variants of FW method and their analyses have been proposed \cite{lan2013complexity,freund2016new,bach2015duality,garber2015faster,lan2017conditional,nesterov2018complexity,braun2019lazifying}, and FW has been extended to stochastic nonconvex \cite{lacoste2016convergence,hazan2016variance,reddi2016stochasticfw,sahu2019towards,balasubramanian2018zeroth,hassani2019stochastic} and online \cite{hazan2012projection,garber2013linearly,lafond2015online,chen2018projection,xie2020efficient,hazan2020faster} settings. 
However these methods provide dimension-free 
\lmocc and \sfocc
only for smooth functions, and further it is known that FW fails to converge if subgradients are used instead of gradients~\cite{nesterov2018complexity}.

{\bf Nonsmooth Frank-Wolfe methods:} \cite{white1993extension} posed an interesting question in the domain of nonsmooth optimization with LMO: 
can
\lmocc
be reduced from the $\Ord(\varepsilon^{-4})$ bound (achieved by \psgd~with PO implemented via LMO: \fwpsgd, see Appendix \ref{sec:fw_proj_subgrad_method}) without increasing 
\focc
significantly. On the lower bound side,  \cite{lan2013complexity} showed that $\Ord(\varepsilon^{-2})$ LMO calls are necessary. On algorithmic side, several randomized smoothing approaches combined with Frank-Wolfe methods were proposed, and can reduce \lmocc to $\Ord(d^{1/2}\varepsilon^{-2})$. But, 
they come at the expense of increased $\Ord(d^{1/2}\varepsilon^{-4})$ FO calls \cite[improving Theorem 5]{lan2013complexity}\footnote{
Needs tightening of \cite[Theorem 5]{lan2013complexity}, by reducing the number of SFO calls per step by a factor of $d^{-1/2}$, i.e.~$T_k = \lceil k d^{-1/2} \rceil$
}. 
If we allow stronger oracles or additional structure in the problem, the complexity can be significantly improved. 
Assuming a stronger than LMO oracle introduced in \cite{white1993extension}, \cite{ravi2019deterministic} shows that 
$\Ord(1/\varepsilon^2)$ 
\lmocc and \focc
are achievable for a special class of problems with low curvatures.  
Another popular setting is when the nonsmooth problem admits a {\em smooth} convex-concave saddle point reformulation \cite{hammond1984solving,cox2017decomposition,pierucci2014smoothing,harchaoui2015conditional,he2015semi,he2015stochastic,gidel2017frank,locatello2019stochastic}. 
Among these the best complexity is achieved by semi-proximal mirror-prox \cite{he2015semi} which uses 
$\Ord(\varepsilon^{-2})$ LMO and $\Ord(\varepsilon^{-1})$ FO calls. 
However, for the general nonsmooth convex optimization problem with LMO, the problem posed by \cite{white1993extension} remained open, and is resolved by our \almo~method that achieves the optimal 
$\Ord(\varepsilon^{-2})$ \lmocc and \focc.

\section{Preliminaries and Notations}
\label{sec:prelims}
We consider Nonsmooth Convex Optimization with FO and PO \eqref{eq:nsco} or LMO \eqref{eq:lmo} accesses.
Let $\cX \subset \vX$ be a closed convex set of diameter $D_\cX \defeq \max_{x_1, x_1 \in \cX} \|x_1-x_2\|$, where $\|\cdot \|$ is the Euclidean norm which corresponds to the inner product $\Ip{\cdot}{\cdot}$. 
Let $\cX$ be enclosed in a closed convex set $\bX \subseteq \vX$ to which it is easy to project, i.e.~$\cX \subset \bX$. For simplicity, let $\bX$ be a Euclidean ball of radius $R \leq D_\cX$ around origin. We can satisfy $R=D_\cX$ by re-centering $\vX$ around any feasible point of $\cX$.
We assume  $f:\bX \to \reals$ to be a proper, lower semi-continuous (l.s.c.), convex Lipschitz function.%
We use $\partial f(x)$ to denote sub-differential of $f$ at $x$, and if $f$ is differentiable we use $\nabla f(x)$ to denote its gradient at $x$. We assume a first-order oracle (FO) can provide access to some subgradient at any point in $\bX$, i.e.~$\fo{x} \in \partial f(x)$.
\begin{definition}\label{def:lipschitzness}
	A function $f: \bX \to \reals$ is $G$-Lipschitz if and only if, $\abs{f(y) - f(x)} \leq G\,\|y - x\|$ for all $x, y \in \bX$. For a convex $f$, this is equivalent to: $\max_{x \in \bX} \max_{g \in \partial f(x)} \|g\| \leq G$.%
\end{definition}
\begin{definition}\label{def:smooth}
	A function $f: \bX \to \reals$ is $\mu$-strongly convex if and only if, $\frac{\mu}{2} \|y-x\|^2+ \Ip{g}{y - x}  + f(x)\leq f(y) $, for all $x,y \in \bX$ and $g\in \partial f(x)$. Similarly, a differentiable function $f: \bX \to \reals$ is said to be $L$-smooth if and only if,
	$f(y) \leq f(x) + \Ip{\nabla f(x)}{y - x} + \frac{L}2 \|y-x\|^2$ for all $x,y \in \bX$. 
\end{definition}
In addition to FO, we also consider problems with stochastic FO (SFO) access, which computes stochastic subgradient of a point $x$ with variance $\sigma^2$, as defined below: 
\begin{equation}
SFO(x):=\hg,\ \text{where}\ \E[\hg \,|\, x] = g \text{ for some } g \in \partial f(x),\ \text{and}\ \E[\|\hg - g\|^2 \,|\, x] \leq \sigma^2.\label{eq:sfo}
\end{equation}

\noindent {\bf Moreau Envelope:} The key idea behind our method is to use ``smoothed'' version of the function via its Moreau envelope~\cite{moreau1965proximite,yosida2012functional} defined below.

\begin{definition}\label{def:Moreau}
	For a proper l.s.c.~convex function $f:{\bX} \to \reals \cup \{\infty\}$ defined on a closed convex set $\bX$ and $\lambda > 0$, its {\em Moreau-(Yosida) envelope} function, $f_\lambda: \bX \to \reals$, is given by 
	\begin{eqnarray}
	f_{\lambda}(x) \;\; = \;\; \min_{x' \in \bX} f(x') + \frac1{2\lambda} \|x - x'\|^2,\;\; \text{ for all $x \in \bX$}\;. 
	\label{eq:envelope}
	\end{eqnarray}
	Furthermore, the~\prox~operator is defined: $\prox_{\lambda f}(x) \defeq \argmin_{x' \in \bX} f(x') + \frac1{2\lambda} \|x - x'\|^2$.
\end{definition}

When $f$ is clear from context, we will  use $\hat{x}_\lambda(x)$ to denote $\prox_{\lambda f}(x)$.
Note that this definition of Moreau envelope is not standard as $x'$ is constrained to $\bX \subseteq \vX$. However, the following lemma (whose proof is in Appendix \ref{lem:moreau-properties-pf}) shows that this Moreau envelope and the \prox~operator still satisfies most useful properties of the standard definition.

\begin{lemma}\label{lem:moreau-properties}
	For a closed convex set $\bX$, a convex proper l.s.c.~function $f:\bX \to \reals \cup \{\infty\}$ and $\lambda > 0$, the following hold for any $x \in \bX$.
	\\
	(a) $\hat{x}_{\lambda}(x)$ is unique and $f(\hat{x}_{\lambda}(x)) \leq f_{\lambda}(x) \leq f(x)$, \\
	(b) $f_{\lambda}$ is convex, differentiable, $1/\lambda$-smooth and  $\nabla f_{\lambda}(x) =  (1 / \lambda) (x - \hat{x}_{\lambda}(x)) $, and,\\
	(c) if $f$ is $G$-Lipschitz continuous, then, $\|\hat{x}_{\lambda}(x)-x\| \leq G\lambda$, and $f(x) \leq f_\lambda(x) + G^2\lambda/2$.
\end{lemma}

This lemma implies that, to find an $\varepsilon$-approximate minima of a nonsmooth  $f$, 
one can instead minimize $f_\lambda$ and achieve a faster convergence by exploiting its smoothness. Concretely, if $f$ is $G$-Lipschitz and $\lambda=O(\varepsilon/G^2)$, and Lemma~\ref{lem:moreau-properties}(c) ensures that solving $f_\lambda$ up to $O(\varepsilon)$ accuracy guarantees 
$O(\varepsilon)$ accuracy in the original minimization of $f$ (Lemma~\ref{lem:joint-gap-to-func-gap}). This insight allows us to design a simple method that can reduce \pocc but at the cost of a higher 
\focc.
Next section starts with this result as a warm-up and then presents our method, which ensures reduced \pocc with optimal
\focc.

\section{Main Results}
We present our main results in this section. We first present the main ideas in Section~\ref{sec:ideas} and then the results for PO and LMO settings in Sections~\ref{sec:mopes} and~\ref{sec:moles} respectively.
\subsection{Main Ideas}\label{sec:ideas}

We are interested in the \nsco problem \eqref{eq:nsco}. As discussed in the previous section, instead of optimizing $f(x)$ over $\cX$, we can instead optimize the Moreau envelope function $\fl(x)$ with $\lambda=O(\epsilon)$ to get $\epsilon$-suboptimality. Since by Lemma~\ref{lem:moreau-properties}, $f_{\lambda}(\cdot)$ is a $1/\lambda$-smooth convex function, a straightforward approach is to iteratively optimize $\fl(x)$ using Nesterov's accelerated gradient descent (AGD)~\cite{nesterov1983method} method. 
But to get gradients of $f_{\lambda}(x)$, we will need to solve the {\em inner problem} \eqref{eq:envelope} approximately. 

A key insight is that since the inner problem does not involve the constraint set $\cX$,  PO calls are not required in inner steps for estimating $\nabla f_{\lambda}(x)$. So the total number of PO calls required is equal to the total number of outer steps in minimizing $f_{\lambda}(x)$, which for Nesterov's AGD is $\Ord({1/\sqrt{\lambda \varepsilon}}) = \Ord({\varepsilon^{-1}})$. We see that this already improves over the $\Ord({\varepsilon^{-2}})$ projections of PGD. However, since $\nabla f_{\lambda}(x)$ needs to be estimated to a good accuracy, the total number of FO calls, including in the inner loop, turns out to be $\Ord({\varepsilon^{-3}})$, which is worse than the optimal $\Ord({\varepsilon^{-2}})$ FO calls of PGD.

Similarly, when we have access to LMO for $\cX$, we could optimize $\fl$ using FW~\cite{frank1956algorithm,jaggi2013revisiting}, with total number of outer steps  $=\Ord({{1}/{\lambda\varepsilon}}) = \Ord({\varepsilon^{-2}})$, and hence the total number of LMO calls is $\Ord({\varepsilon^{-2}})$. However, this again leads to suboptimal $\Ord({\varepsilon^{-4}})$ FO calls. We can improve the 
\focc
to $\Ord({\varepsilon^{-3}})$ by using the conditional gradient sliding algorithm~\cite{lan2016conditional} instead of FW method, but this is still worse than the optimal $\Ord({\varepsilon^{-2}})$ FO calls.

In order to achieve optimal number of FO calls, we directly optimize the Moreau envelope through the following joint optimization.
\begin{align} \label{eq:joint_moreau_opt}
\min_{x \in \cX, x' \in \bX} \; [\Psi_\lambda(x, x') \defeq  f(x') + \psi_\lambda(x,x') ] \;\; \mbox{ where } \;\; \psi_\lambda(x,x') \defeq \frac1{2\lambda} \|x'-x\|^2,
\end{align}
where the function $\Psi_\lambda: \bX \times \bX \to \reals$ is convex in the joint variable $(x, x')$. The main advantage of this new form is that, this is a composite optimization problem with a nonsmooth part (corresponding to $f(x')$) and a $2/\lambda$-smooth part (corresponding to $({1}/{2\lambda}) \norm{x'-x}^2$) with the constrained variable $x \in \cX$ only appearing in the smooth part.
Now, by the following lemma, an approximate minimizer of $\Psi_\lambda$, is also an approximate minimizer of the Moreau envelope $f_\lambda$, and further if $\lambda = \varepsilon/G^2$,
it is also an approximate minimizer of the original function $f$.
A proof is provided in Appendix~\ref{sec:joint-gap-to-func-gap-pf}.
\begin{lemma}\label{lem:joint-gap-to-func-gap}
	Under the same assumptions as in Lemma~\ref{lem:moreau-properties}, let $\cX \subseteq \bX$ be a convex subset and $\Psi_\lambda$ be defined as in \eqref{eq:joint_moreau_opt}. Then, $(i)$ $\min_{x \in\cX} \min_{x \in \bX} \Psi_\lambda(x,x') = \min_{x \in \cX} f_\lambda(x) \leq \min_{x \in \cX} f(x)$ , and 	$(ii)$ for any random vectors $(x_\varepsilon, x_\varepsilon') \in \cX \times \bX$, $\E[f(x_\varepsilon)] - G^2 \lambda/2\leq \E[f_\lambda(x_\varepsilon)] \leq \E[\Psi_{\lambda}(x_\varepsilon, x_\varepsilon')]$.
\end{lemma}

Our algorithm essentially solves~\eqref{eq:joint_moreau_opt} using Gradient Sliding~\cite{lan2016gradient} and Conditional Gradient Sliding~\cite{lan2016conditional} frameworks, which are optimal for minimizing composite problems of the form~\eqref{eq:joint_moreau_opt} for the PO and LMO settings respectively. The resulting algorithm for PO setting, called~\aproj~is given in Algorithm~\ref{algo:envelope_subgrad_method_proj}. The algorithm for LMO setting, called~\almo~is presented in Algorithm~\ref{algo:envelope_subgrad_method_lmo}. The only difference between~\aproj~and~\almo~is that~\almo~uses FW to compute approximate projections while~\aproj~uses exact projections. Finally, our algorithms extend straightforwardly to the case of stochastic subgradients through a stochastic first order oracle (SFO) and the resulting bounds depend on the variance of SFO in addition to the Lipschitz constant of $f(\cdot)$.
\begin{algorithm}[t]
	\caption{\aproj: MOreau Projection Efficient Subgradient method
}
	\label{algo:envelope_subgrad_method_proj}
	\DontPrintSemicolon %
	\SetKwFunction{PROJSTEP}{Approx-Proj}
	\SetKwFunction{PROXSLIDE}{Prox-Slide}	
	\KwIn{$f$, 
		$\cX$, $\bX$,
		$G$, $D_{\cX}$, $x_0$, $K$, 
		$\tilde{D}$, $c'$, $\lambda$, 
	}
	
	\SetKwProg{Fn}{}{:}{}
	{
		Set ${\px}_0 ' = {\dx}_0 ' = {\px}_0 = {\dx}_0 = x_0$\;
		\For{$k = 1, \ldots, K$} {
			Set $\beta_{k} = \frac{4}{\lambda k}\;, \;\gamma_{k} = \frac{2}{k+1}\,, \text { and } 
			T_{k} = \Big \lceil{ \frac{(4G^2 + \sigma^2) \lambda^2 K k^2}{2\tilde{D}}}\Big\rceil$
			\nllabel{algo_line:envelope-subgrad-param-set_proj}\;
			
			Set $({\mx}_{k}, {\mx}_{k}') =\left(1-\gamma_{k}\right) \cdot ({\px}_{k-1},{\px}_{k-1}')+\gamma_{k} \cdot ({\dx}_{k-1}, {\dx}_{k-1}') $
			\nllabel{algo_line:midpoint_compute_proj}\;
			
			Set ${\dx}_{k} =\cP_\cX\big({\dx}_{k-1} - \frac1{\beta_k} \cdot \nabla_{\mx_k} \Psi_{\lambda}(\mx_k,\mx_k')\big)$ \eqref{eq:nsco} \tcp*{\color{violet} Note $\nabla_{\mx_k} \Psi_{\lambda}(\mx_k,\mx_k') = \frac{\mx_k - \mx_k'}{\lambda}$} 			\nllabel{algo_line:dual_compute_proj}

			Set $\left({\dx}_{k}', {\tdx}_{k}'\right)=\text{\PROXSLIDE}\big(\nabla_{\mx_k'} \psi_{\lambda}(\mx_k,\mx_k'), {\dx}_{k-1}', \beta_{k}, T_{k}\big)$ 
			\tcp*{\color{violet}  $\nabla_{\mx_k'} \psi_{\lambda}(\mx_k,\mx_k')=\frac{{\mx}_{k}' - {\mx}_{k}}\lambda$}
			\nllabel{algo_line:dual_compute_prime_proj}

			Set $({\px}_{k}, {\px}_{k}') =\left(1-\gamma_{k}\right) \cdot ({\px}_{k-1},{\px}_{k-1}') +\gamma_{k} \cdot ({\dx}_{k}, {\tdx}_{k}')$
			\nllabel{algo_line:primal_compute_proj}
		}
		\KwOut{$({\px}_{K}, {\px}'_{K})$}
	}
	
\vspace*{5pt}
	\Fn{\PROXSLIDE{$g$, 
			$u_0$,
			$\beta$, $T$}{\color{violet}\ // Approx. resolve {\small ${\prox}_{f/\beta}\big(u'_0 -g/\beta \big)$}\cite{lan2016gradient}}\nllabel{algo_line:proxslide_proj}
		}{
		Set $\widetilde{u}_{0} = u_{0}$\;
		\For{$t = 1, \ldots, {T}$} {
			Set $\theta_{t} = \frac{2(t+1)}{t(t+3)}$, 
			 $\ \ \hg_{t-1} = \sfo{u_{t-1}}$ \eqref{eq:sfo}
			 \nllabel{algo_line:sfo_proj} \nllabel{algo_line:prox-slide-param-set_proj}\;
			Set 
			$\widehat{u}_{t} = u_{t-1} - \frac{1}{(1+t/2)\beta} \cdot (\hg_{t-1} + \beta (u_{t-1} -(u_{0} - g/\beta)))$
			\linebreak \tcp*{\color{violet}
				subgradient method step for $\phi(u) \defeq f(u) + \frac\beta2 \|u-\big(u_{0} - \frac{g}{\beta} \big)\|^2$}
			\nllabel{algo_line:sgd_step_proj}
			Set $u_t = \widehat{u}_t \cdot \min\left(1,R/\norm{u_t}\right)$\tcp*{\color{violet} projection of $\widehat{u}_t$ onto $\bX$: $\mathcal{P}_\bX(\u_t)$} \nllabel{algo_line:prox-slide-proj_proj}
			Set $\widetilde{u}_{t}=\big(1-\theta_{t}\big) \cdot \widetilde{u}_{t-1}+\theta_{t} \cdot u_{t}$
			\nllabel{algo_line:prox-slide-avg_proj}\;
		}
		\KwRet{$(u_{T}, \widetilde{u}_{T})$}\;
	}
\end{algorithm}

\subsection{MOreau Projection Efficient Subgradient (\aproj) method} 
\label{sec:mopes}

A pseudocode of our algorithm~\aproj~is presented in Algorithm~\ref{algo:envelope_subgrad_method_proj}. At a high level,~\aproj~is an inexact Accelerated Proximal Gradient method (\apgd)~\cite{nesterov2013gradient,beck2009fast} scheme which tries to implement Nesterov's AGD algorithm on $\Psi_{\lambda}(x,x')$. Now, standard AGD updates for solving $\min_{x\in \cX, x'}\Psi_{\lambda}(x,x')$, {\it if $\Psi_{\lambda}$ were smooth} are:   %
\begin{empheq}[box=\widefbox]{align}
\hspace*{-20pt}
{
	\begin{aligned}
	\beta_{k} &\gets  {4}/{\lambda k}\;, \;\gamma_{k} \gets {2}/{(k+1)} \\ 
	({\mx}_{k}, {\mx}_{k}') &\gets \left(1-\gamma_{k}\right) ({\px}_{k-1},{\px}_{k-1}')+\gamma_{k} ({\dx}_{k-1}, {\dx}_{k-1}') \\
	{\dx}_{k} &\gets  \cP_\cX \left({\dx}_{k-1} - \nabla_{{\mx}_{k}} \Psi_{\lambda}({{\mx}_{k}},  {\mx}_{k}')/\beta_k \right),\ \  {\dx}_{k}' \gets {\dx}_{k-1}' - \nabla_{{\mx}_{k}'} \Psi_{\lambda}({{\mx}_{k}}, {\mx}_{k}')/\beta_k, \\
	({\px}_{k}, {\px}_{k}') &\gets \left(1-\gamma_{k}\right) ({\px}_{k-1},{\px}_{k-1}') +\gamma_{k} ({\dx}_{k}, {\dx}_{k}').
	\end{aligned}
}
\label{eq:apgd_update_main}%
\end{empheq}
\aproj essentially implements the above updates, but as $\Psi_{\lambda}$ is  nonsmooth in $x'$, we use ~\prox~steps for the $x'$ variable instead of the GD steps. The \prox~ step---$\prox_{f/\beta_k}\big({\dx}_{k-1}' - \nabla_{{\mx}_{k}'} \psi_{\lambda}({{\mx}_{k}}, {\mx}_{k}')/\beta_k\big)$---is  implemented via \PROXSLIDE procedure (see ~\Cref{algo_line:dual_compute_prime_proj}), which is the standard subgradient method applied to a strongly convex function $\phi$ (see~\Cref{algo_line:sgd_step_proj}). Now, \proxslide procedure outputs two points $({\dx}_k',{\tdx}_k')$ which are the final and average iterates, respectively, of the subgradient method, This achieves optimal 
\focc
by exploiting strong convexity of $\phi$. If we were to use only the average of the iterates, the 
\focc
would increase by a factor of $\Ord(\varepsilon^{-1})$ (see the failed attempt in Appendix~\ref{sec:concept}).

Note that \aproj needs only a PO call \& no FO call in \Cref{algo_line:dual_compute_proj}, and only a FO/SFO call in \Cref{algo_line:sfo_proj}. Therefore, we bound below, the total number of PO calls $K$ and the number of FO/SFO calls $K\cdot T$.

\begin{theorem}
	\label{thm:envelope-subgrad-method-cor1}
	Let $f:\bX \to \reals$ be a $G$-Lipschitz continuous proper l.s.c.~convex function equipped with a SFO with variance $\sigma^2$, and $\cX\subseteq \bX = B(0, R)$ be some convex subset equipped with a projection oracle $\cP_{\cX}$ and contained inside the Euclidean ball of radius $R$ around origin.
	If we run~\aproj~(Algorithm~\ref{algo:envelope_subgrad_method_proj}) with inputs $\lambda = \varepsilon/{G^2}$, $\tilde{D} = c \|x_0 - x^*\|^2$ and $K = \lceil {2 \sqrt{(10+8c)} G \|x_0 - x^*\|}/{\varepsilon}\rceil$ for any absolute constant $c > 0$ and $x^* \in \argmin_{x \in \cX} f(x)$, then, using $\Ord({\frac{G \|x_0 - x^*\|}{\varepsilon}})$ PO calls and $\Ord(\frac{(G^2+\sigma^2)\|x_0 - x^*\|^2}{\varepsilon^2})$ FO calls, it outputs $x_K$ satisfying
	$f\left({\px}_{K}\right)- \min_{x \in \cX} f(x) \; \leq \;\; \varepsilon$.
\end{theorem}
{\bf Remarks}: Note that 
\focc
is same as that of \psgd (up to constants) while \pocc is significantly better. A natural open question is if \pocc can be further reduced.
Also, \aproj requires querying of SFO/FO at $u_{t-1}$ which is not necessarily in $\cX$ but is always in $\bX$ (\Cref{algo_line:sfo_proj}).
Recall from Section~\ref{sec:prelims} that $\bX$ is a Euclidean ball of radius $R \leq D_\cX$ around origin.
Being able to query SFO/FO in $\bX$ seems like a mildly stricter requirement than the standard requirement of querying on $\cX$ only, but for most practical problems this seems feasible. Even if $f$ is unknown outside of $\cX$, theoretically we could work with its convex extension to the entire space, which remains $G$ Lipschitz (see Section~\ref{sec:conclusion}).
Also, notice that the guarantee only depends on the diameter $D_\cX$ of the constraint set $\cX$ and not the radius $R$ of the enclosing set $\bX$. This is so because the first-order method only depends on the distance from initial point $(x_0, x_0')$ to the desired solution $(x^*, x^*)$, which is $\Ord(\|x_0 - x^*\|) = \Ord(D_\cX)$, as $x_0' = x_0$.
Finally, for simplicity of exposition, we provide desired suboptimality $\epsilon$ as an input to~\aproj%
--in practice, we can remove this assumption by using standard doubling trick \cite[Algorithm 6]{thekumparampil2019efficient}. 

\iftrue
{%
	See Appendix~\ref{sec:envelope-subgrad-method-cor1-pf} for a detailed proof of Theorem\ref{thm:envelope-subgrad-method-cor1}. 
	Here we provide a short proof sketch for the theorem to showcase the main analysis techniques used by the full proof.
	At a high level, our proof uses a potential function~\cite{bansal2017potential} for analyzing \apgd, combines it with Proposition~\ref{prop:main} which provides a fast convergence guarantee on \proxslide iterates, and then applies standard \apgd proof techniques \cite{tseng2008accelerated} to obtain the final result. 
	\vspace{-.4em}
	\begin{proofsketch}
		In this sketch we only consider the deterministic FO. Consider the following potential (Lyapunov) function, which is a slight modification of the standard AGD potential~\cite{bansal2017potential}, for any $x \in \cX$.
		\begin{align}
		\Phi_k \defeq k(k+1) (\Psi_\lambda({\px}_{k},{\px}_{k}') - \Psi_\lambda(x,x)) + ({4}/{\lambda}) (\| {\dx}_{k} - x\|^2 + \tau_{k+1} \|  {\dx}_{k}' - x\|^2) %
		\end{align}
		where $\tau_k \defeq {(T_{k}+1)(T_{k}+2)}/{T_{k}(T_{k}+3)}$. We will prove that this potential satisfies the descent rule: $\Phi_{k} \leq \Phi_{k-1} + k \eta_k'$, for some error $\eta_k'$. Using the fact that $\Psi_\lambda$ is a sum of two convex functions: $2/\lambda$-smooth quadratic $\psi_\lambda$ and $G$-Lipschitz $f$, and standard analysis techniques for AGD we can get
		\begin{align}
		k&(k+1)\Psi_\lambda({\px}_{k},{\px}_{k}') 
		\leq\; k(k-1) \Psi_\lambda({\px}_{k-1},{\px}_{k-1}') + 2k \psi_\lambda(x, x) +
		\nonumber \\&
		2k [ \Ip{\nabla_{k,x}}{{\dx}_{k}} + ({\beta_k}/{2}) \|{\dx}_{k} - {\dx}_{k-1}\|^2] + 
		2k [ \phi_k({\tdx}_{k}') - \phi_k(x) + ({\beta_k}/{2})  \|x - {\dx}_{k-1}'\|^2] \label{eq:envelope-subgrad-method-pfsk-eq1}
		\end{align}
		where we use the short-hands $\nabla_{k} \defeq \nabla \psi_\lambda({\mx}_{k}, {\mx}_{k}')$ and $\phi_k(x') \defeq f\left(x'\right) + \Ip{\nabla_{k,x'} }{x'} + ({\beta_k}/{2})\|x'-{\dx}_{k-1}'\|^{2}$. Next, using definition of projection 
		${\dx}_k$,
		we bound the third term in the RHS of \eqref{eq:envelope-subgrad-method-pfsk-eq1} as
		\begin{align}
		2k[\Ip{\nabla_{k,x} }{x -{\dx}_{k}} + ({\beta_k}/{2})\left\|{\dx}_{k} - {\dx}_{k-1}\right\|^{2}] &\leq 2k ({\beta_k}/{2})[\left\|{\dx}_{k-1} - {x}\right\|^{2} - \left\|{\dx}_{k} - {x}\right\|^{2}]\,. \label{eq:envelope-subgrad-method-pfsk-eq2}
		\end{align}
		The fourth term in the RHS of \eqref{eq:envelope-subgrad-method-pfsk-eq1} corresponds to the $\varepsilon$-approximate resolution of the $\prox_{f/\beta_k}$ operator through the \proxslide procedure (\Cref{algo_line:dual_compute_proj}), whose output satisfies the following guarantee.
		\begin{proposition}[informal version of Proposition~\ref{prop:general-thm-prop-2}] \label{prop:main}
			Output of \proxslide satisfies
			\begin{align*}
			\phi_k({\tdx}_k') - \phi_k (x) + \frac{\beta_k}2 \|{\dx}_{k}'-x\|^2 \leq \frac{\beta_k}2 (\tau_{k} - 1) [\|{\dx}_{k-1}'-x\|^2 - \|{\dx}_{k}'-x\|^2] + \frac{16\,G^2}{\beta_k T_k }. 
			\end{align*}
		\end{proposition}
		The above lemma guarantees the optimal $O(1/T_k)$ convergence rate for the strongly convex minimization problem: $\min_{{\dx}' \in \bX} \phi_k({\dx}')$, corresponding to the proximal operator.
		By combining the inequalities \eqref{eq:envelope-subgrad-method-pfsk-eq1} and \eqref{eq:envelope-subgrad-method-pfsk-eq2} and the proposition we get: $\Phi_{k} \leq \Phi_{k-1} + k  \Ord(G^2/\beta_k T_k)$. Now, using Lemma~\ref{lem:joint-gap-to-func-gap}, and setting $x = x^*$, $\lambda = \frac\varepsilon{G^2}$, $\beta_k = \frac4{\lambda k}$, $T_k = \Ord(k)$ and $K = \Theta(\frac{G \|x_0 -x^*\|}\varepsilon)$ we get
		\begin{align*}
		f(x_K) - f(x^*) &\leq \Psi_\lambda({\px}_{K},{\px}_{K}') - \Psi_\lambda(x^*,x^*) + G^2 \frac\lambda2 \nonumber \\
		&\leq \frac{8\|x_0 - x^*\|^2 }{\lambda K(K+1)} + \frac{\sum_{k=1}^K k\,{16 G^2}/{\beta_k T_k}}{\lambda K(K+1)} + G^2 \frac\lambda2 = \Ord(\varepsilon) \nonumber \\
		\end{align*}
		Therefore, the total number of PO calls made is $K = \Ord(G \|x_0 -x^*\|/\varepsilon)$ and the total number of FO calls made is $\sum_{k=1}^K T_k = \Ord(K^2) = \Ord(G^2 \|x_0 -x^*\|^2/\varepsilon^2)$.
	\end{proofsketch}
	
}
\fi

\subsection{MOreau Linear minimization oracle Efficient Subgradient (\almo) method}\label{sec:moles}
		\begin{algorithm}%
			\caption{\almo: MOreau Linear minimization oracle Efficient Subgradient method}
			\label{algo:envelope_subgrad_method_lmo}
			\DontPrintSemicolon %
			\SetKwFunction{FWProj}{FW-Based-Projection}	
			\setcounter{AlgoLine}{-1}
			\nonl Use the same steps as \aproj~(Algorithm~\ref{algo:envelope_subgrad_method_proj}), but replace \Cref{algo_line:dual_compute_proj} with: \;
			\setcounter{AlgoLine}{4}
			\SetKwProg{Fn}{}{:}{}
			{
				Set $\dx_k = \FWProj(\dx_{k-1} - \frac{1}{\beta_k} \cdot \nabla_{\mx_{k}} \Psi_{\lambda}(\mx_k,\mx_k'),\, \dx_{k-1},\, \big\lceil \frac{7 K D_\cX^2}{c' \tilde{D}} \big\rceil)$
			}\;
\vspace*{5pt}
			\setcounter{AlgoLine}{15}
			\Fn{\FWProj{$z$, $u_0$, $\hat{T}$}
			}{\tcp{\color{violet} $\hat{T}$ steps of standard Frank-Wolfe for $\min_{u \in \cX} \|{u-z}\|^2$}
				\For{$t = 1, \ldots, \hat{T}$} {
					Set $s_t = \lmo{u_{t-1}-z}$ \;%
					Set $u_{t} = ((t-1) \cdot u_{t-1} + 2 \cdot s_t)/(t+1)$
				}
				\KwRet{$u_{\hat{T}}$}\;
			}
		\end{algorithm}

We now present our results for the LMO setting. A pseudocode of our algorithm,~\almo, is presented in Algorithm~\ref{algo:envelope_subgrad_method_lmo}.~\almo~does exactly the same steps as in~\aproj~(Algorithm~\ref{algo:envelope_subgrad_method_proj}), except that the projection in \Cref{algo_line:dual_compute_proj} of~\aproj~is estimated using the LMO and Frank-Wolfe algorithm. At the outer-step $k$, the output ${\dx}_{k}$ of~\FWProj, which uses $\hat{T} = \Ord(1/\varepsilon)$ LMO calls to approximately project, satisfies the following bound on the projection problem's Wolfe dual gap~\cite{jaggi2013revisiting}:
\begin{align}\label{eq:envelope-subgrad-FWdual-lmo}
\max _{ s \in \cX} \; \beta_k \Ip{\dx_{k} - \big(\dx_{k-1}-({1}/{\beta_k}) \cdot \nabla_{{\mx}_{k}} \Psi_{\lambda}(\mx_{k},\mx_{k}')\big)}{{\dx}_k - s} \leq {4c'\tilde{D}}/{\lambda K k} \,
\end{align}
In practice we can use the above condition as a stopping criterion for \FWProj.
The following theorem, a proof of which is in Appendix \ref{sec:envelope-subgrad-method-cor2-pf}, provides the convergence guarantee.
\begin{theorem}\label{thm:envelope-subgrad-method-cor2}
	Let $f:\bX \to \reals$ be a $G$-Lipschitz continuous proper l.s.c.~convex function equipped with an SFO with variance $\sigma^2$, and $\cX\subseteq \bX = B(0, R)$ be some convex subset of diameter $D_\cX$ equipped with an LMO and contained inside the Euclidean ball of radius around origin.
	If we run~\almo~(Algorithm~\ref{algo:envelope_subgrad_method_lmo}) with inputs $\lambda = \varepsilon/{G^2}$, $\tilde{D} = c D_\cX$, $K =\lceil {2 \sqrt{10+8c(1+c')} G \|x_0 - x^*\|}/{\varepsilon}\rceil$, for some absolute constants $c$,$c' > 0$ and $x^* \in \argmin_{x \in \cX} f(x)$, then, using $\Ord({\frac{G^2 D_\cX^2}{\varepsilon^2}})$ LMO calls and $\Ord(\frac{(G^2+\sigma^2)D_\cX^2}{\varepsilon^2})$ FO calls it outputs $x_K$ satisfying
	$f\left({\px}_{K}\right)- \min_{x \in \cX} f(x) \; \leq \;\; \varepsilon$.
\end{theorem}
{\bf Remarks}: 
Thus our algorithm obtains the optimal $\Ord(\varepsilon^{-2})$ dimension independent 
\focc and \lmocc
for general nonsmooth functions~\cite{lan2013complexity}. Similar to \aproj, here also, we require FO/SFO of $f$ to be well-defined in $\bX$. If $f$ is a maximum of smooth convex functions, then we can get similar 
\pocc
by applying min-max saddle point approaches~\cite{he2015semi}. But even for such functions, it is non-trivial to extend saddle point approaches to stochastic FO, which is important in practice. In contrast, our result matches the optimal 
\focc
(on all key parameters) of unconstrained stochastic-\psgd method.

\section{Applications}
 
 \label{sec:applications}
We first explain the gain of \aproj~in practical applications. %
One of the main applications of our method is Empirical Risk Minimization (ERM) with nonsmooth loss functions. 
For a nonsmooth loss  $f_i$  for the $i$th training example in a set of $n$ examples, the general form of ERM is:
\begin{align}\label{eq:erm}
	\min_{x \in \cX \subseteq \reals^d} \frac{1}{n} \sum_{i=1}^n f_i(x) %
	\; \text{ For example, }\quad
	\min_{{X \in \R^{m \times p} ;  \norm{X}_{\rm nuc} \leq r}} \frac{1}{n} \sum_{i=1}^{n} \max(0,1-b_{i}\iprod{X}{A_i}) \;,
\end{align}
which is known as 
the {\em low rank SVM}~\cite{wolf2007modeling,wang2013efficient,razzak2019sparse} as the nuclear norm constraint induces low rank solutions. 
As the cost of a single PO call involves a full SVD on a potentially full rank $X$, 
\aproj significantly improves over the competing baseline as we showcase in Fig.~\ref{fig:aproj_imagewoof}. 
There are numerous examples of ERMs with costly POs to a nuclear norm ball (e.g.~max-margin collaborative filtering \cite{srebro2005maximum}), to an $\ell_1$ norm ball
(e.g.~sparse SVM \cite{bradley1998feature,zhu20041,bach2012optimization}), 
and to a large number of linear constraints (e.g.~robust classification \cite{ben2012efficient}).
One notable example is {\em SVM with hard constraints} on a subset of  the training data, so that
some predictions are constrained to be always accurate \cite{nguyen2014efficient} (See Appendix~\ref{sec:hard_SVM}). 

In all these examples, PO calls can be more costly than FO calls, making \aproj~attractive.  
In comparison, popular {\it accelerated proximal point methods}, such as FISTA \cite{beck2009fast}, cannot handle general nonsmooth losses. 
The standard {\it projected subgradient methods} suffer from $\Ord(\varepsilon^{-2})$ \pocc. {\it Mirror descent}~\cite{nemirovsky1983problem} may give better $d$ dependence, but it too requires $O(\varepsilon^{-2})$ (proximal) operations.

Now, several nonsmooth loss functions have a special structure where they can be written as a {\it smooth minimax problem}. Such (stochastic) problems can be  solved using 
$\Ord_\varepsilon(\varepsilon^{-1})$ (S)FO and PO calls \cite{nesterov2005smooth}. However,  the resulting complexity 
scales up with the dimension $d$ or the number of samples $n$.
Thus the \pocc of the minimax formulations becomes inefficient (even with variance reduction~\cite{palaniappan2016stochastic,carmon2019variance}), whenever $n$ or $d$ gets large. In the deterministic setting, each step of the optimization problem requires gradient of the entire empirical risk function, so for problems with large $n$ and small $\varepsilon$, total time complexity can be significantly higher than \aproj. See Appendix~\ref{sec:ell1_SVM} for exact complexities.

Further, beyond ERM, nice minimax representations might not always exist. 
For example, in reinforcement learning/optimal control setting, $f$ could be an
(already trained) {\it input-convex neural network}~\cite{amos2017input,chen2018optimal} approximating the Q-function over a continuous constrained action space~\cite{chen2020input}.

For several of the above examples, LMO might be preferred if it is significantly more efficient than a PO call e.g., for high-dimensional low rank SVM, a LMO call only requires computing top singular vector, as opposed to full SVD required by a PO.
Further, LMO-based methods have an additional benefit of preserving the desired structure of the solution, such as sparse  and low rank structures~\cite{clarkson2010coresets}. 
This makes \almo~particularly attractive, for example, 
in differentially private collaborative filtering \cite{jain2018differentially}, where structured updates lead to improved privacy guarantees. 
In 
Appendix~\ref{sec:applications_details},
we present the details of some these examples, and give analytical comparisons to competing methods.

\section{Empirical Results}

\label{sec:exps}

\begin{figure}
	\begin{minipage}{.5\textwidth}
	\centering
	\includegraphics[width=0.95\textwidth]{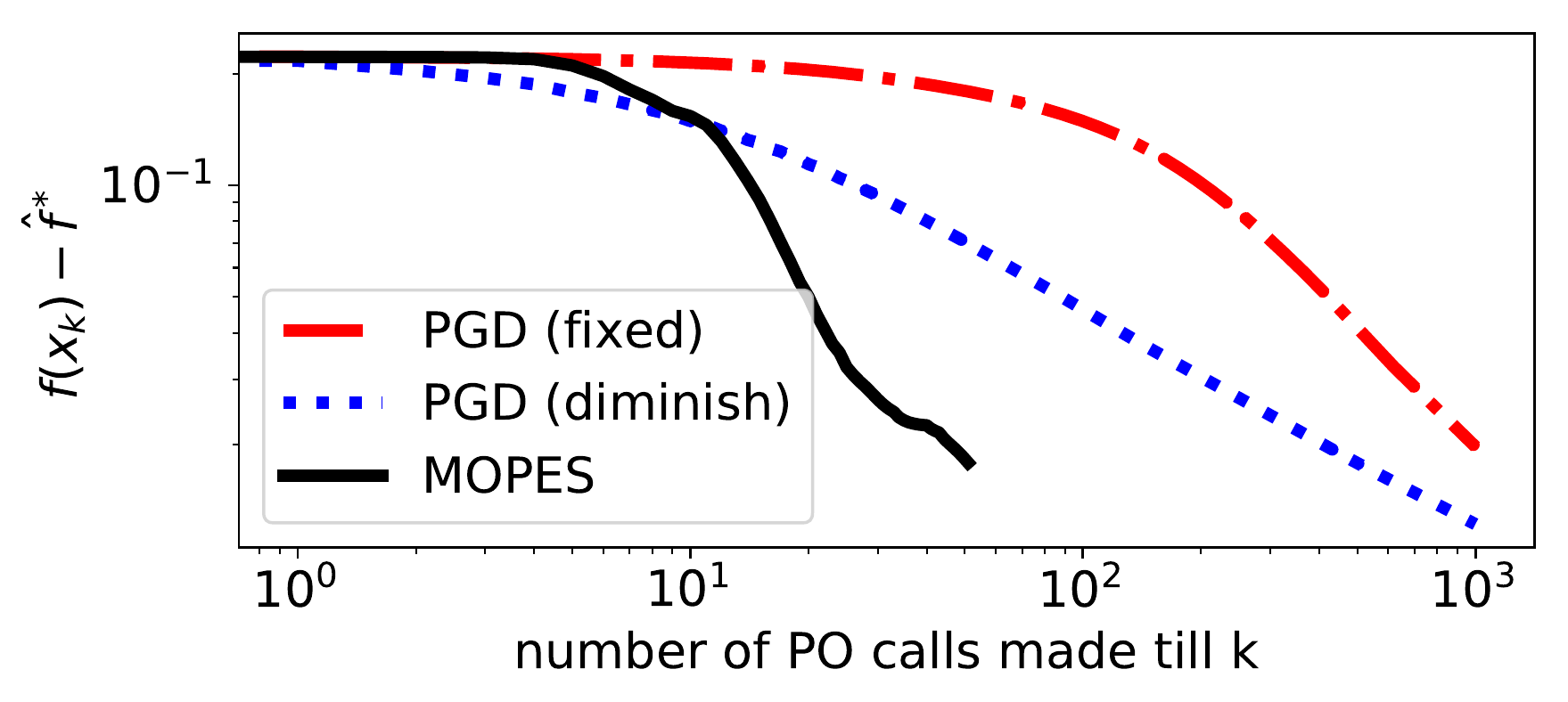}
	\includegraphics[width=0.95\textwidth]{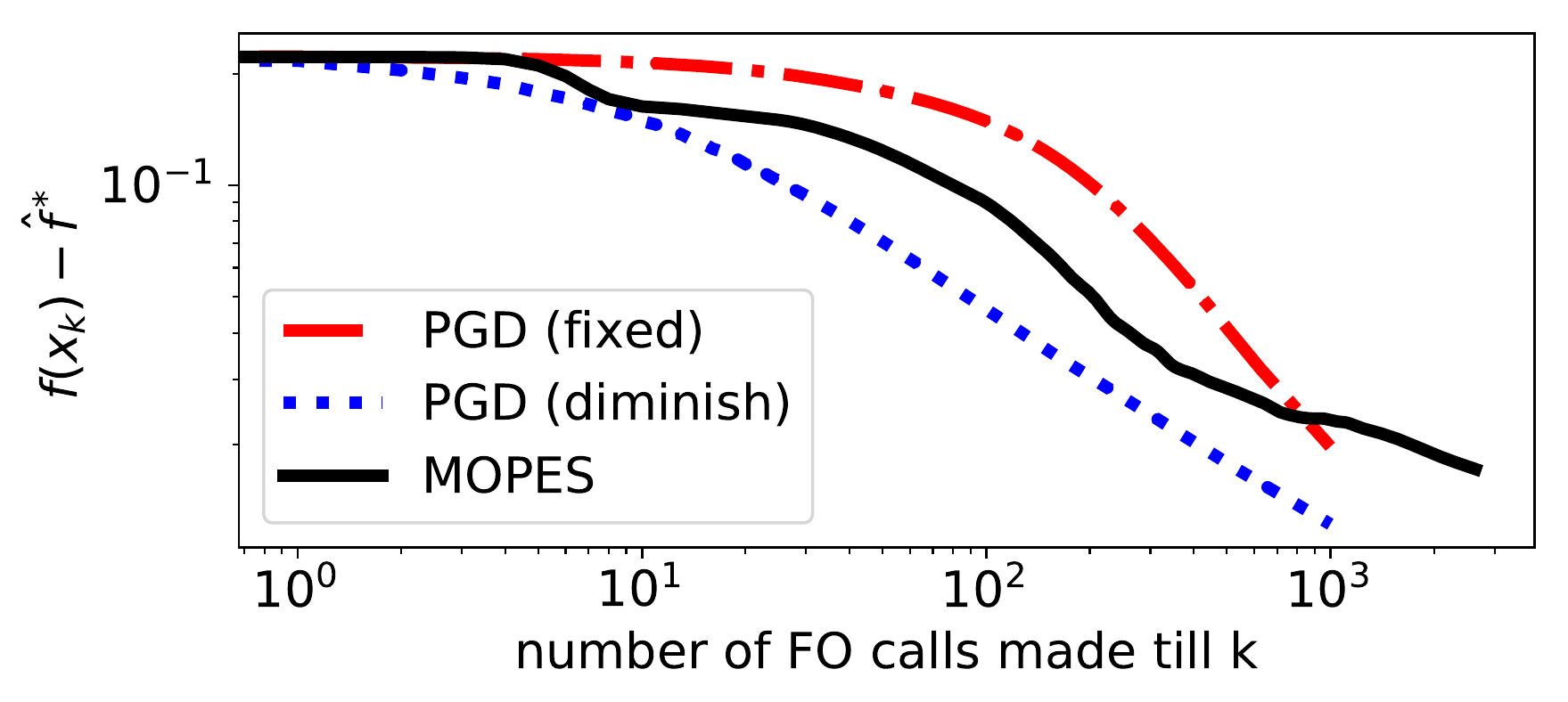}	
	\vspace*{-10pt}
	\caption{
		\aproj uses significantly fewer PO calls and comparable number of FO calls than \psgd
	}
	\label{fig:aproj_imagewoof}
\end{minipage}\hspace*{5pt}
	\begin{minipage}{.5\textwidth}
		\centering
		\includegraphics[width=0.95\textwidth]{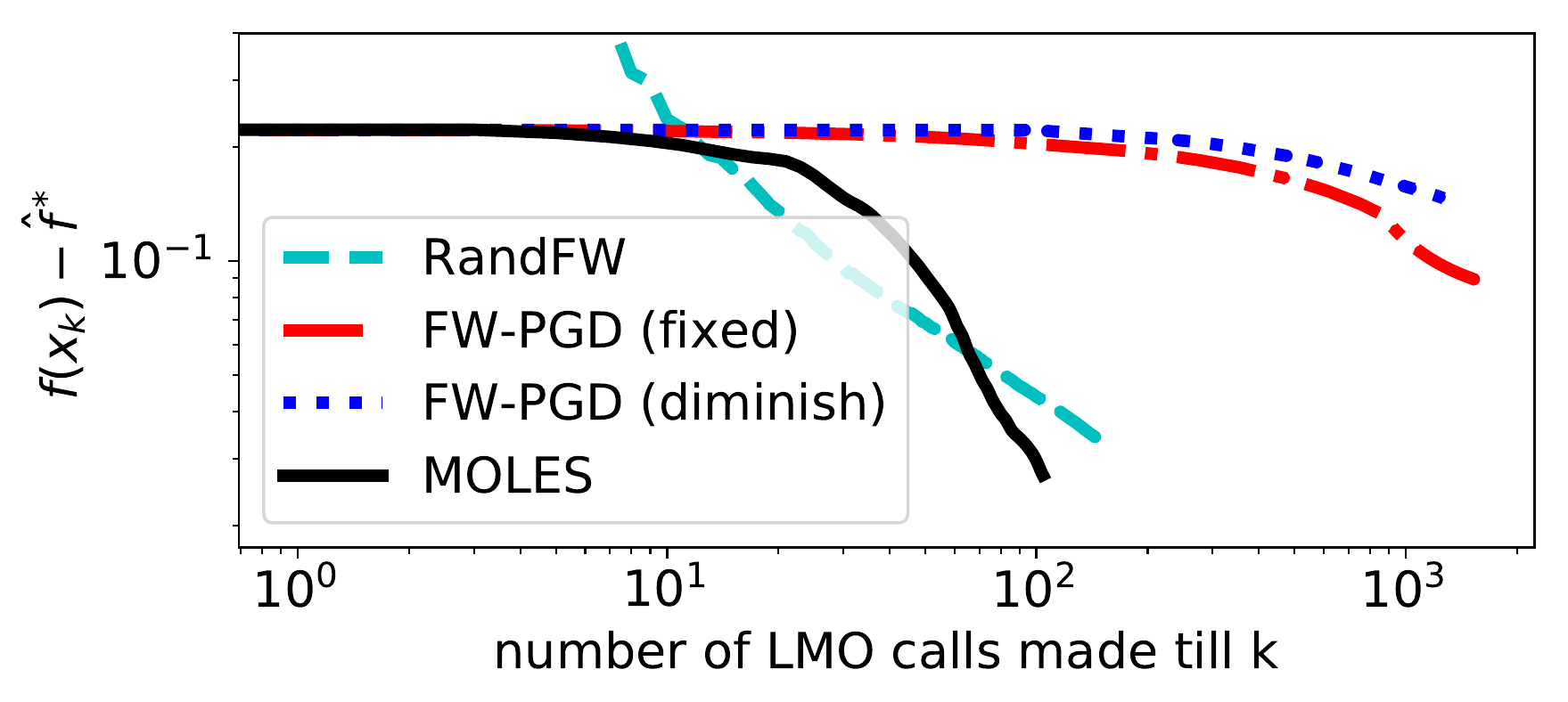}
		\includegraphics[width=0.95\textwidth]{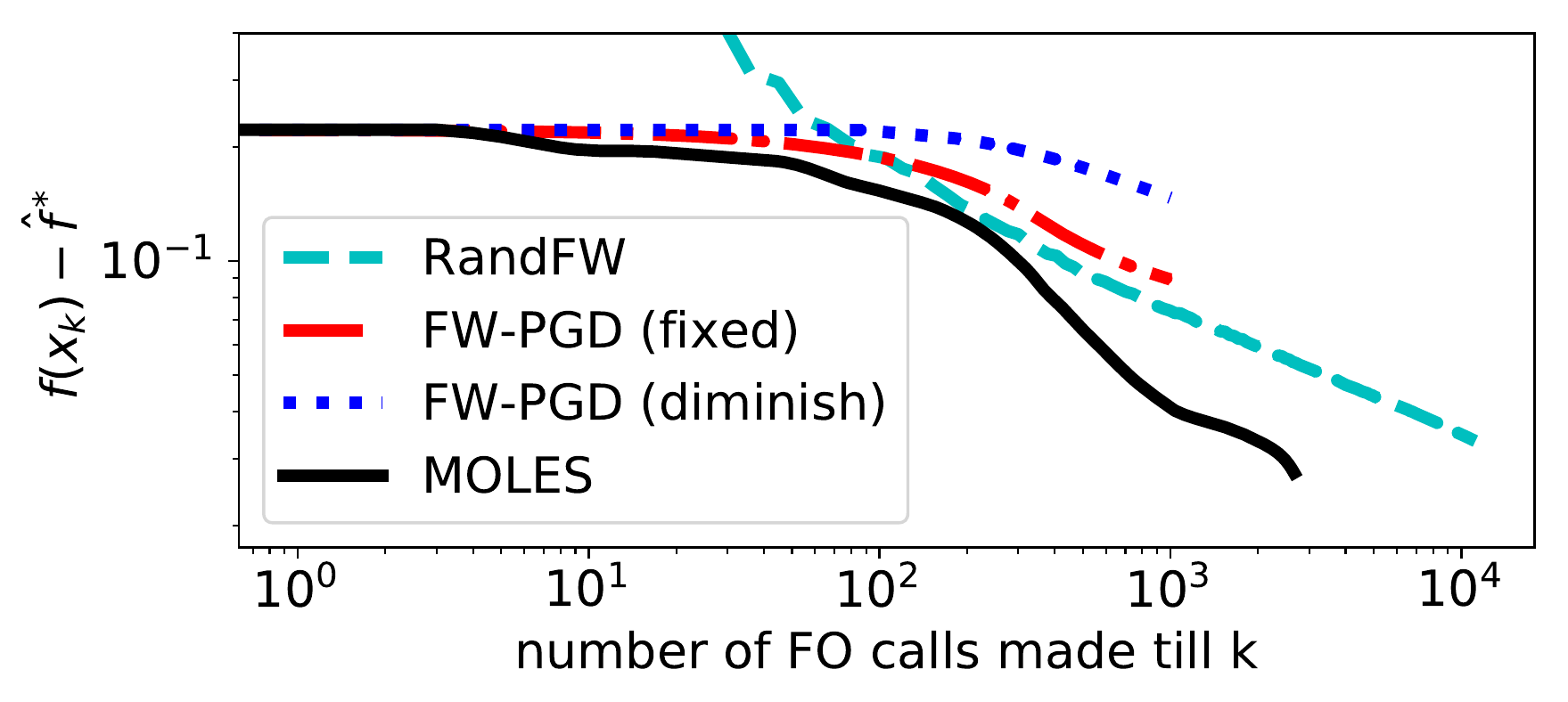}	
		\vspace*{-10pt}
		\caption{
			\almo uses fewer LMO calls and similar number of FO calls than \fwpsgd and \randfw
		}
		\label{fig:almo_imagewoof}
	\end{minipage}
	
\end{figure}

We experimentally evaluate\footnote{Code for the experiments is available at \url{https://github.com/tkkiran/MoreauSmoothing}} \aproj (Algorithm~\ref{algo:envelope_subgrad_method_proj}) and \almo (Algorithm~\ref{algo:envelope_subgrad_method_lmo}) methods on a low rank SVM problem \cite{wolf2007modeling} of the form \eqref{eq:erm} on a subset of the Imagewoof 2.0 dataset~\cite{How2019Imagenette}.
The training data contains $n=400$ samples $\{(A_i, y_i)\}_{i=1}^n$ where $A_i$ is a $224\times224$ grayscale image labeled using $y_i \in \{0,1\}$. Note that the effective dimension is $d = 50176$. We use $r=0.1$ as nuclear norm ball radius of $\cX$.
First, we compare the PO and FO efficiencies of \aproj~with those of \psgd with a fixed and \psgd with a diminishing stepsize. 
In Figure~\ref{fig:aproj_imagewoof} we plot the mean (over 10 runs) sub-optimality gap: $f(x_k) - \hat{f}^*$,  of the iterates against the number of PO (top) and FO (bottom) calls, respectively, used to obtain that iterate. 
Next, we compare the LMO and FO efficiencies of \almo~with those of \fwpsgd (see Algorithm~\ref{algo:fw_proj_subgrad_method} in Appendix~\ref{sec:fw_proj_subgrad_method}) and Randomized Frank-Wolfe (\randfw) \cite[Theorem 5]{lan2013complexity} methods with a fixed and diminishing stepsizes. %
In Figure~\ref{fig:almo_imagewoof} we plot the mean (over 10 runs) sub-optimality gap: $f(x_k) - \hat{f}^*$,  of the iterates against the number of LMO (top) and FO (bottom) calls, respectively, used to obtain that iterate. 
In both these plots, while \aproj/\almo and baselines have comparable 
\focc,
\aproj/\almo is significantly more efficient in the number of PO/LMO calls, matching our Theorems~\ref{thm:envelope-subgrad-method-cor1} and \ref{thm:envelope-subgrad-method-cor2}. 
As the nuclear norm ball has a non-trivial projection/LMO, \pocc/\lmocc will dominate the total run-time as $m$ becomes larger for $X\in{\mathbb R}^{m\times m}$. Note that matrix mirror descent~\cite{kulis2009low} 
would also require $O(\varepsilon^{-2})$ SVD based proximal operations. 
We provide additional experimental details in Appendix~\ref{sec:exps_details}.

\section{Conclusion}
\label{sec:conclusion}
We study a canonical problem in  optimization: minimizing a nonsmooth Lipschitz continuous convex function over a convex constraint set. 
We assume that  the function is accessed with a first-order oracle (FO) and the set is accessed with either a projection oracle (PO) or a linear minimization oracle (LMO). 
In this general setting, we address the fundamental question of reducing the number of accesses to the function and the set.   
When using projections, we introduce \aproj, and show that it finds an $\varepsilon$-suboptimal solution with $\Ord(\varepsilon^{-2})$ FO calls and $\Ord(\varepsilon^{-1})$ PO calls. 
This is optimal in the number of FO calls and significantly improves over competing methods in the number of PO calls (see Table~\ref{tab:results}). 
When using linear minimizations, we introduce \almo, and show that it finds an $\varepsilon$-suboptimal solution with $\Ord(\varepsilon^{-2})$ FO and LMO calls. This is optimal in both the number of PO and the number of LMO calls. This resolves a question left open since \cite{white1993extension} on designing the optimal Frank-Wolfe type algorithm for nonsmooth functions.

The two properties we need of the superset $\bX \supseteq \cX$ are that (a) it is easy to project onto $\bX$ and (b) $f$ is $G$-Lipschitz on $\bX$. In our paper, we choose $\bX$ to be a Euclidean ball (which is easy to project to) but any other choice of $\bX$ which satisfies the above properties works just as well. For example, if $f$ is Lipschitz everywhere, we can set $\bX = \vX$ and ignore the explicit projection to $\bX$ in \cref{algo_line:prox-slide-proj_proj} of Algorithm \ref{algo:envelope_subgrad_method_proj}.
However, even if $f$ is $G$-Lipschitz inside the constraint $\cX$, $f$ could (i) have unbounded Lipschitz constant, or (ii) be undefined just outside of $\cX$. Thus an $\bX$ satisfying our requirements may not exist. In our experiments, we do not explicitly project onto $\bX$ (\cref{algo_line:prox-slide-proj_proj}) but still observed that $\|x_k - x_k'\| = \Ord(G\lambda) = \Ord(\varepsilon)$ and small, which implies that the iterates $x_k'$ are close to $\cX$. This hints that we may only need Lipschitzness over a much smaller set, say $\cX + B(0, \Ord(G\lambda))$, but we do not have a proof for this conjecture. Theoretically, we can work around the issue (ii) above by minimizing the convex extension $f_\cX : \mathbb{R}^d \rightarrow \mathbb{R}$ of the function $f$ from the set $\cX$, defined as $f_\cX(x') \defeq \max_{x \in \cX} \max_{g \in \partial f(x)} f(x) + \Ip{g}{x'-x}$. The extension $f_\cX$ has the same value as $f$ inside $\cX$ and is $G$-Lipschitz everywhere. Therefore the minimization problems $\min_{x \in \cX} f(x)$ and $\min_{x \in \cX} f_\cX(x')$ are equivalent. However, it is not clear if we can estimate/approximate the gradients of $f_\cX$ efficiently. We did not find any relevant prior work and leave this question for future work.

Another possible direction of future work is developing $\varepsilon$-horizon oblivious algorithms, where we need not fix $K$ and $\varepsilon$ a priori. In our experiments, we observed that varying $\lambda$ according to $\lambda_k = O(\frac{D_\cX}{Gk})$ and $\beta_k = \frac{4}{\lambda_k k}$ works just as well as fixing it. However, obtaining any guarantee for this scheme seems challenging and again seems to require the bound $\|\px_k - \px_k'\| = \Ord(G \lambda_k)$.

\newpage

\section*{Broader Impact}
As this is foundational research that is theoretical in nature, it is hard to predict any foreseeable societal consequence.

\bibliography{optimization}

\newpage
\appendix
\section*{Appendix}

\section{Supplementary results}
\label{sec:supp_result}

{\color{black}
\subsection{Intuition behind the design of \aproj and a failed attempt}
\label{sec:concept} 
In this section we study the main ideas behind the design of \aproj method through a failed attempt.
Only for the section, for simplicity, we assume that $\bX$ is the whole vector space, and $f$ is $G$ Lipschitz in $\bX$.
Recall that we want to solve the problem~\eqref{eq:joint_moreau_opt},
\begin{align}
\min_{x \in \cX, x' \in \bX} [ \Psi_{\lambda}(x,x') = \psi_{\lambda}(x,x') + f(x') ] \,.
\end{align}
Notice that this is a composite objective which is a sum of a $2/\lambda$-smooth function $\psi_\lambda$ and a nonsmooth function $f$. This implies that, if we have access to the proximal operator (recall Definition \ref{def:Moreau}) for $f$
\begin{align}\label{eq:prox_op}
\prox_{f/t}(z) \defeq \arg\min_{x \in \bX} f(x) + \frac{t}{2}\|x-z\|^2\,,
\end{align}
then theoretically we can solve this problem using accelerated proximal gradient algorithm (\apgd)~\cite{beck2009fast,nesterov2013gradient,tseng2008accelerated}, which has the following update rule
\begin{empheq}[box=\widefbox]{align}
\begin{aligned}
\beta_{k} &\gets  {4}/{\lambda k}\;, \;\gamma_{k} \gets {2}/{(k+1)} \\
({\mx}_{k}, {\mx}_{k}') &\gets \left(1-\gamma_{k}\right) ({\px}_{k-1},{\px}_{k-1}')+\gamma_{k} ({\dx}_{k-1}, {\dx}_{k-1}') \\
{\dx}_{k} &\gets  \cP_\cX \left({\dx}_{k-1} - \nabla_{x} \psi_{\lambda}({{\mx}_{k}}, {\mx}_{k}')/\beta_k \right) \\
{\dx}_{k}' &\gets \prox_{f/\beta_k} \left({\dx}_{k-1}' - \nabla_{x'} \psi_{\lambda}({{\mx}_{k}}, {\mx}_{k}')/\beta_k\right) \\
({\px}_{k}, {\px}_{k}') &\gets \left(1-\gamma_{k}\right) ({\px}_{k-1},{\px}_{k-1}') +\gamma_{k} ({\dx}_{k}, {\dx}_{k}')
\end{aligned}
\,,\label{eq:apgd_update}\tag{\apgd}
\end{empheq}
for some stepsize $1/\beta_k$ and iterate weight $\gamma_k$. 

{
Notice that this update rule is different from the standard accelerated schemes, because the latter either first update the primal variables $(x,x')$ and then extrapolate the dual variables $(z,z')$~\cite{nesterov1998introductory} or simultaneously update them both~\cite{nesterov1983method,lan2011primal}, whereas \eqref{eq:apgd_update}, which is fashioned along the lines of~\cite{tseng2008accelerated}, first updates $(z,z')$ using proximal\footnote{projection $\cP_\cX$ could be considered as a proximal step for the $0$-$\infty$ indicator function for the set $\cX$} step and then extrapolates these to update $(x,x')$. Advantage of \cite{tseng2008accelerated} over the standard rule are three fold; former only needs one proximal step per variable (as opposed to two in~\cite{nesterov1983method,lan2011primal}) per iteration (which makes it practically faster), or keeps the dual and middle iterates $z_k$, $y_k$ feasible (as opposed to~\cite[(2.2.17)]{nesterov1998introductory}), and can easily handle stochastic FO and constraints~\cite{lan2012optimal}.
Another reason for the choice, which will be evident later on, is that, our update rule can simultaneously provide the optimal complexity for the smooth $\psi_\lambda$ and nonsmooth $f$ parts of the composite function, $\Psi_\lambda$~\eqref{eq:joint_moreau_opt}~\cite{lan2016gradient}.
}

With the right choice of $\beta_k$, $\gamma_k$, \eqref{eq:apgd_update} can find an $\varepsilon$-approximate solution to the problem~\eqref{eq:joint_moreau_opt}, $({\px}_K, {\px}_K')$, in ${\cal O}({\sqrt{2/\lambda\varepsilon}})$ steps. Now if we choose $\lambda = {\cal O}({\varepsilon})$, we can show that ${\px}_K$ is also an ${\cal O}({\varepsilon})$ solution of our original nonsmooth constrained problem~\eqref{eq:nsco}. This is formalized in the Lemma \ref{lem:joint-gap-to-func-gap}.
Thus applying \eqref{eq:apgd_update} on \eqref{eq:joint_moreau_opt} with $\lambda=\varepsilon/G^2$ gives us an $\varepsilon$ solution to the original problem~\eqref{eq:nsco} using only $K = {\cal O}({G/\varepsilon})$ projections, which is a significant improvement over the ${\cal O}({G^2/\varepsilon^2})$ PO calls used by the standard subgradient method. 

For a general $G$-Lipschitz convex function $f$, we cannot solve $\prox_{f/\beta_k}$ exactly, and hence we resort to some approximate solution. We emphasize here that it is not immediately evident that we can implement an inexact $\prox$ operator, and still maintain that the total number of FO calls used by this inexact \apgd~method match the optimal lowerbound ${\cal O}({G^2/\varepsilon^2})$~\cite{nemirovsky1983problem,nesterov1998introductory}. Perhaps surprisingly, this is achieved using the Gradient Sliding method~\cite{lan2016gradient}, which proposes a specific form of an inexact \apgd. Note that the constrained variable $x \in \cX$ is not an input to the nonsmooth part $f$ of $\Psi_{\lambda}$, which means that approximately resolving the $\prox$ operator $\prox_{f/t}$ does not require any projection.

As an intermediate algorithm we first present \eqref{eq:iapgd_update}, 
which is derived from \eqref{eq:apgd_update} but replaces the proximal update of ${\dx}_{k}'$ with an inexact resolution for prox operator ${\rm prox}_{f/\beta_k}$ up to an approximation error of $\delta$. Notice that $\delta=0$, implies that $z_k$ is an exact resolution of the operation $\prox_{f/\beta_k}(\hat{\dx}_{k}')$. 
The specific choice of the approximation error is important, as other notions of approximation error of the proximal operator in the context of \apgd~(such as those in~\cite{schmidt2011convergence}) do not explicitly control the distance $\|x'-{\dx}_{k}'\|^2$ which is crucial for our guarantee. Although with $\delta=\varepsilon$,
\eqref{eq:iapgd_update} would require $\Ord(1/\varepsilon^3)$ FO calls, we provide the details of its analysis, in the next theorem, as it showcases some of the ideas behind the design of our our main algorithm (Algorithm~\ref{algo:envelope_subgrad_method}).
\begin{empheq}[box=\widefbox]{align}
\begin{aligned}
&\text{Use the update rule of \eqref{eq:apgd_update}, but replace  $\prox$ step by the following:} \nonumber \\
&\text{find ${\dx}_{k}'$ satisfying the following for all $x' \in \bX$} \nonumber \\
&\frac{\beta_k}{2} \|x' - {\dx}_{k}'\|^2 + f({\dx}_{k}') + \frac{\beta_k}{2} \|{\dx}_{k}' -\hat{\dx}_{k}' \|^2 \;\leq\;  f({x}')  + \frac{\beta_k}{2} \|{x}' - \hat{\dx}_{k}'\|^2  + \delta\;,\\
&\text{ where }\hat{\dx}_{k}' = ({\dx}_{k-1}' - \nabla_{x'} \psi_{\lambda}({{\mx}_{k}}, {\mx}_{k}')/\beta_k) 
\end{aligned}
\label{eq:iapgd_update}\tag{\iapgd}
\end{empheq}
\begin{theorem}\label{thm:concept-envelope-subgrad-method-thm}
Let $f:\bX \to \reals$ be a $G$-Lipschitz continuous convex function and $\cX\subseteq \bX$ be any convex  set with a diameter $D_\cX$ and projection oracle $\cP_\cX$. If we choose $\lambda = \varepsilon/G^2$ and $\delta=\varepsilon$, then after $K = \Ord({G D_\cX/\varepsilon})$ iterations of the \ref{eq:iapgd_update} update rule, initialized with $y_0' = y_0 = x_0' = x_0$,
finds ${\px}_K \in \cX$ satisfying $f({\px}_K) - \min_{x\in\cX} f(x) \leq {\varepsilon}$. 
Further, $\Ord({G^2 D_\cX^2/\varepsilon^2})$ iterations of a standard subgradient method ensures the condition in \eqref{eq:iapgd_update}. 
In total, this algorithm requires $\Ord({G^3D_\cX^3/\varepsilon^3})$ FO calls and $\Ord(GD_\cX/\varepsilon)$ PO calls.
\end{theorem}
{\bf Remarks:} Even though \ref{eq:iapgd_update} only achieves a FO-CC of $\Ord(1/\varepsilon^3)$, the main take away from this result should be that with this right choice for approximate resolvent of the \prox operator $\prox_{f/\beta_k}$ \eqref{eq:iapgd_update}, we can achieve $\Ord(1/\varepsilon)$ PO-CC. This is exploited by our \aproj method (Algorithm \ref{algo:envelope_subgrad_method_proj}) which uses a more efficient \proxslide procedure~\cite{lan2016gradient}
to approximately resolve the \prox~operator, so as to obtain a PO-CC of $\Ord(1/\varepsilon)$ while still maintaining the optimal FO-CC o $\Ord(1/\varepsilon^2)$.
\begin{proof}[Proof of Theorem \ref{thm:concept-envelope-subgrad-method-thm}]
Now consider the following potential (Lyapunov) function from \cite{bansal2017potential} for arbitrary $x \in \cX$ and $x' \in \bX$: 
\begin{align}
\Phi_k \defeq k(k+1) (\Psi_\lambda({\px}_{k},{\px}_{k}') - \Psi_\lambda(x,x')) + ({4}/{\lambda}) \| ({\dx}_{k}, {\dx}_{k}') - (x,x')\|^2
\end{align}
We will prove that this potential satisfy the following approximate descent condition $\Phi_{k} \leq \Phi_{k-1} + k \varepsilon$ as follows. Notice that by $2/\lambda$-smoothness and convexity of $\psi_\lambda$
\begin{align}
\psi_\lambda({\px}_{k},{\px}_{k}') \leq\; &\psi_\lambda({\mx}_{k},{\mx}_{k}') + \Ip{\nabla_k }{({\px}_{k},{\px}_{k}') - ({\mx}_{k},{\mx}_{k}')} + \frac1{\lambda}  \|({\px}_{k},{\px}_{k}') - ({\mx}_{k},{\mx}_{k}')\|^2 \nonumber \\
\leq\; &(1-\gamma_k)\psi_\lambda({\px}_{k-1},{\px}_{k-1}') + \nonumber \\
&\gamma_k [\psi_\lambda({\mx}_{k},{\mx}_{k}') + \Ip{\nabla_k }{({\dx}_{k},{\dx}_{k}') - ({\mx}_{k},{\mx}_{k}')} + \frac{\gamma_k}{\lambda}  \|({\dx}_{k},{\dx}_{k}') - ({\dx}_{k-1},{\dx}_{k-1}')\|^2]
\end{align}
where we use the shorthand $\nabla_k \defeq [\nabla_{k,x}^T \nabla_{k,x'}^T]^T \defeq [\nabla_{x} \psi_\lambda({\mx}_{k},{\mx}_{k}')^T \; \nabla_{x'} \psi_\lambda({\mx}_{k},{\mx}_{k}')^T]^T$. Now combining this with $f({\px}_{k}') \leq (1-\gamma_k) f({\px}_{k-1}') +  \gamma_k f({\dx}_{k}')$ (convexity) and $\gamma_k/\lambda \leq \beta_k/2$ we get that
\begin{align}
&k(k+1)\Psi_\lambda({\px}_{k},{\px}_{k}') \nonumber \\
\leq\; &k(k-1) \Psi_\lambda({\px}_{k-1},{\px}_{k-1}') + 2k \psi_\lambda({\mx}_{k},{\mx}_{k}') +
2k [ \Ip{\nabla_{k,x}}{{\dx}_{k} - {\mx}_{k}} + \frac{\beta_k}{2}  \|{\dx}_{k} - {\dx}_{k-1}\|^2] \nonumber \\
&2k [ f({\dx}_{k}') + \Ip{\nabla_{k,x'}}{{\dx}_{k}' - {\mx}_{k}'} + \frac{\beta_k}{2}  \|{\dx}_{k}' - {\dx}_{k-1}'\|^2] \nonumber \\
\leq\; &k(k-1) \Psi_\lambda({\px}_{k-1},{\px}_{k-1}') + 2k \psi_\lambda({\mx}_{k},{\mx}_{k}') +
\nonumber \\&
2k [ \Ip{\nabla_{k,x}}{x - {\mx}_{k}} + \frac{\beta_k}{2}  (\|{\dx}_{k-1} - x\|^2 - \|{\dx}_{k} - x\|^2)] \nonumber \\
&2k [ f(x') + \Ip{\nabla_{k,x'}}{x' - {\mx}_{k}'} + \frac{\beta_k}{2}  (\|{\dx}_{k-1}' - x'\|^2 - \|{\dx}_{k}' - x'\|^2) + \varepsilon] \nonumber \\
\leq\; &k(k-1) \Psi_\lambda({\px}_{k-1},{\px}_{k-1}') + 2k \Psi_\lambda(x,x') +
\nonumber \\&
(4/\lambda) (\|({\dx}_{k-1} ,{\dx}_{k-1}') - (x,x')\|^2 - \|({\dx}_{k},{\dx}_{k}') - (x,x')\|^2) + k \varepsilon\,,
\end{align}
where the second inequality uses the definition of projection and the $\varepsilon$-approximate resolution of the proximal operator~\eqref{eq:iapgd_update}, and the last inequality again uses convexity of $\psi_\lambda$.
This proves that $\Phi_{k} \leq \Phi_{k-1} + k \varepsilon$, which 
directly implies that
\begin{align}
&\Psi_\lambda({\px}_{K},{\px}_{K}') - \Psi_\lambda(x,x') \leq \frac{4(\|x_0 - x\|^2 + \|x_0 - x'\|^2)}{ \lambda K(K+1)} + \frac1{K(K+1)} \sum_{k=1}^{K} k \varepsilon %
\end{align}
Setting $x' = x$, choosing $\lambda = \varepsilon/G^2$ and $K=\Ord({GD_{\cX}/\varepsilon})$ gives us $\Psi_\lambda({\px}_{k},{\px}_{k}') - f(x) \leq \varepsilon/2$. Then by Lemma~\ref{lem:joint-gap-to-func-gap} we get that $f({\px}_{k}) - \min_{x \in \cX} f(x) \leq \varepsilon$. For each inner problem the standard (unconstrained) proximal subgradient method applied on $\min_{x' \in \bX} f(x') + (\beta_k/2) \|x' - ({\dx}_{k-1}' - \nabla_{k,x'}/\beta_k)\|^2$, initialized with $x_0$ (for ease of argument), can achieve this error using $\Ord( {G^2 \|x_0 - \hx_\lambda(x)\|^2/\varepsilon^2}) = \Ord( {G^2 D_{\cX}^2/\varepsilon^2})$ FO calls~(Lemma~\ref{lem:subgradient_method}, in Appendix~\ref{sec:subgradient_method}). Thus the algorithm uses totally $\Ord({GD_\cX/\varepsilon})$ projections and $\Ord({G^3 D_{\cX}^3/\varepsilon^3})$ subgradients.
\end{proof}

}


\section{Supporting results}
\label{sec:lemmas} 

\subsection{Proximal Subgradient method} \label{sec:subgradient_method}

\begin{lemma}[proximal subgradient descent]\label{lem:subgradient_method}
Consider the regularized optimization problem
\begin{align}
\min_{u} [f_{\beta,x}(u) \defeq f(u) + (\beta/2) \|u-x\|^2]
\end{align}
 and the proximal subgradient method's update rule
\begin{align}
u_{t+1} &= \argmin_{u} [F_{t}(u) \defeq \Ip{g_t}{u-x} + (1/{2\eta}) \|u-u_t\|^2 + \beta/2\|u-x\|^2] \nonumber \\
&= u_t - (\eta/(1+\eta \beta))(g_t + \beta (u_t - x))
\end{align}
where $g_t \in \partial f(u_t)$ and $\eta$ is the effective stepsize. Now, if $\eta = 2\,G^2\|u_0 - u\|/\sqrt{T}$ and $\widetilde{u}_T = \frac1T \sum_{t=0}^{T-1} u_{t+1}$, then for any $u$
\begin{align}
\frac\beta2 \|\widetilde{u}_{T} - u\|^2 + f_{\beta,x}(\widetilde{u}_{T}) - f_{\beta,x}(u) &\leq \frac{2\,G\,\|u_0 - u\|}{\sqrt{T}}
\end{align}
\end{lemma}
\begin{proof}
Let $u$ be an arbitrary feasible point. By convexity and $G$-Lipschitzness of $f$,
\begin{align}
f(u_{t+1}) - f(u) &=  f(u_{t+1}) - f(u_t) + f(u_{t}) - f(u) \nonumber \\
&\leq \Ip{g_{t+1}}{u_{t+1} - u_t} + \Ip{g_{t}}{u_{t} - u}\nonumber \\
&= \Ip{g_{t}}{u_{t+1} - u_t} + \Ip{g_{t+1} - g_t}{u_{t+1} - u_t} + \Ip{g_{t}}{u_{t} - u}\nonumber \\
&\leq \Ip{g_t}{u_{t+1} - u} + 2G\,\|u_{t+1} - u_t\|\,, \label{eq:prox_subgrad_eq1}
\end{align}
As $u_{t+1}$ is the minimizer of a $(\beta+1/\eta)$-strong convexity update objective $F_{t}$ and since, we get that
\begin{align}
\big(\frac\beta2 + \frac1{2\eta}\big) \|u_{t+1} - u\|^2 + F_{t}(u_{t+1}) &\leq F_{t}(u) \label{eq:prox_subgrad_eq2}
\end{align}
Now summing up \eqref{eq:prox_subgrad_eq1}, sand \eqref{eq:prox_subgrad_eq2} we get
\begin{align}
\frac\beta2 \|u_{t+1} - u\|^2 + f_{\beta,x}(u_{t+1}) - f_{\beta,x}(u) &\leq \frac1{2\eta}(\|u_{t}-u\|^2 - \|u_{t+1}-u\|^2) + \nonumber \\
&\;\;\; 2G\,\|u_{t+1} - u_{t}\| - \frac1{2\eta}\|u_{t+1} - u_{t}\|^2  \nonumber \\
&\leq \frac1{2\eta}(\|u_{t}-u\|^2 - \|u_{t+1}-u\|^2) +  {2\,G^2 \eta} \nonumber \\
\implies \frac1T \sum_{t=0}^{T-1} \frac\beta2 \|u_{t+1} - u\|^2 + f_{\beta,x}(u_{t+1}) - f_{\beta,x}(u) &\leq \frac1{2\eta T} (\|u_{0}-u\|^2 - \|u_{T}-u\|^2) + {2\,G^2 \eta}  \nonumber \\
\frac\beta2 \|\widetilde{u}_{T} - u\|^2 + f_{\beta,x}(\widetilde{u}_{T}) - f_{\beta,x}(u) &\leq 
\end{align}
where the second inequality follows from $ax - x^2/2b \leq a^2b/2$, the third inequality is obtained by summing over $t=0,\ldots,T-1$, and the third inequality uses Jensen's inequality. Choosing $T = 2\,G\,\|u_0 - u\|/\sqrt{T}$, we get the desired result
\begin{align}
\frac\beta2 \|\widetilde{u}_{T} - u\|^2 + f_{\beta,x}(\widetilde{u}_{T}) - f_{\beta,x}(u) &\leq \frac{2\,G\,\|u_0 - u\|}{\sqrt{T}}
\end{align}
\end{proof}

\subsection{Frank-Wolfe projected subgradient method \fwpsgd~(Algorithm~\ref{algo:fw_proj_subgrad_method}) }
\label{sec:fw_proj_subgrad_method}
Here we provide the details of the Frank-Wolfe based projected subgradient method (Algorithm~\ref{algo:fw_proj_subgrad_method}) used in the experiments. The main idea is to use some competitive LMO based method to approximate the projection step in the standard projected subgradient method. The following theorem gives some guarantees for the output of the Algorithm~\ref{algo:fw_proj_subgrad_method}.

\begin{algorithm}[t]
	\caption{Frank-Wolfe projected subgradient method using LMO}
	\label{algo:fw_proj_subgrad_method}
	\DontPrintSemicolon %
	\KwIn{$f$, 
		$\cX$, $G$, $D_{\cX}$, $x_0$, $K$,
	}
	
	\SetKwProg{Fn}{}{:}{}
	{
		\For{$k = 0, \ldots, K-1$} {
			Set $\hg_{k} = \sfo{x_{k}}$ \;
			Using any competitive LMO based algorithm (e.g.~Frank-Wolfe method~\cite{frank1956algorithm} or ${\rm CndG}$ procedure~\cite[Algo.~1]{lan2016conditional}), approximately solve the projection problem
			\begin{align}		
			x_{k+1} \approx \argmin _{x \in \cX} \Ip{\hg_{k}}{x} +\frac1{2 \alpha_{k}} \|x-x_{k}\|^2 = \argmin _{x \in \cX} \frac1{2\alpha_{k}} \|x-(x_{k}-\alpha_{k} \cdot \hg_{k})\|^2\,,\label{eq:fwpgd-FWproj}
			\end{align}
			ensuring that the Wolfe duality gap~\cite{jaggi2013revisiting} of the above problem at $u_{\Pi}$ satisfies
			\begin{align}\label{eq:fwpgd-FWdual}x
			\max _{ s \in \cX} \Ip{\hg_{k} +1/\alpha_k \left(x_{k+1} - {x_k}\right)}{{x}_{k+1} - s } \leq \eta_k\,
			\end{align}		
		}
		\KwOut{$\bar{x}_K = \frac{\sum_{k=0}^{K-1} \alpha_k x_k}{\sum_{k=0}^{K-1} \alpha_k}$}
	}
\end{algorithm}

\begin{theorem}\label{thm:fw_proj_subgrad_method-thm}
	Let $f:\bX \to \reals$ be a $G$-Lipschitz continuous proper l.s.c.~convex function, and
	$\cX\subseteq \bX$ be some closed convex subset of $\reals^d$ with diameter $D_\cX$. Then after $K$ iterations, the Algorithm~\ref{algo:fw_proj_subgrad_method}	projection tolerance $\eta_k = (G^2+\sigma^2) \alpha_k$, stepsize $\alpha_k = \frac{D_\cX}{2\sqrt{G^2 + \sigma^2}\sqrt{K}}$ and outputs $\bar{x}_K \in \cX$ satisfying 
	\begin{align}
	\E[f\left(\bar{x}_{K}\right)] - f(x^*) \leq &\frac{2 \sqrt{G^2 + \sigma^2} D_\cX}{\sqrt{K}}
	\end{align}
	Further, the algorithm uses $K$ SFO calls and $O(K^2)$ LMO calls.
\end{theorem}
\begin{proof}
	Using the Wolfe duality gap guarantee we get that for any $x \in \cX$
	\begin{align}
	\Ip{\hg_{k} +\frac1{\alpha_k} \left(x_{k+1} - {x_k}\right)}{{x}_{k+1} - x } \leq \eta_k\,.
	\end{align}
	By rearranging the terms above we get that
	\begin{align}
	\Ip{\hg_{k} }{{x}_{k} - x }  &\leq \frac1{2\alpha_k} (\|x_k - x\|^2 - \|x_{k+1} - x\|^2) - \frac1{2\alpha_k} \|x_{k+1} - x_k\|^2 + \Ip{\hg_{k} }{{x}_{k} - x_{k+1} } + \eta_k \nonumber \\
	&\leq \frac1{2\alpha_k} (\|x_k - x\|^2 - \|x_{k+1} - x\|^2) - \frac1{2\alpha_k} \|x_{k+1} - x_k\|^2 + \|\hg_{k}\| \|{x}_{k} - x_{k+1} \| + \eta_k \nonumber \\
	&\leq \frac1{2\alpha_k} (\|x_k - x\|^2 - \|x_{k+1} - x\|^2) + \frac{\alpha_k}2  \|\hg_{k}\|^2 + \eta_k\,,
	\end{align}
	where the last inequality uses the fact that $-(a/2) z^2 + b z \leq b^2/2a$ for all $a, b, z \in \reals$. 
	Next, multiplying by $\alpha_k$ and summing the above inequality over $k=0,\ldots, K-1$ and dividing by $\sum_{{k'}=0}^{K-1} \alpha_{k'}$ we get
	\begin{align}
	\sum_{k=0}^{K-1} \alpha_k \Ip{\hg_{k} }{{x}_{k} - x }  &\leq \frac1{2} (\|x_0 - x\|^2 - \|x_{K} - x\|^2) + \sum_{k=0}^{K-1}  {\alpha_k^2} (\frac{\|\hg_{k}\|^2}2 + \frac{\eta_k}{\alpha_k})\,,
	\end{align}
	Now taking expectation w.r.t.~all the stochasticity in $\{\hg_{k}\}_{k=0}^{K-1}$ on both sides and using, the towering conditional expectation property $\E[a] = \E[\E[a \,|\, x_k]]$, and $\E[\hg_{k}\,|\, x_k] = g_k \in \partial f(x_k)$ and $\E[\|\hg_{k}\|^2 \,|\, x_k] \leq 2 (G^2 + \sigma^2)$ we get 
	\begin{align}
	\sum_{k=0}^{K-1} \alpha_k \E[\Ip{g_{k} }{{x}_{k} - x }]  &\leq \frac1{2} \|x_0 - x\|^2 + \sum_{k=0}^{K-1}  {\alpha_k^2} ((G^2 + \sigma^2) + \frac{\eta_k}{\alpha_k})\,,
	\end{align}
	Next diving by $\sum_{{k'}=0}^{K-1} \alpha_{k'}$, using convex affine lower bound of $f$ at $x_k$ and Jensen's inequality we get
	\begin{align}
	\sum_{k=0}^{K-1} \frac{\alpha_k}{\sum_{k'=0}^{K-1} \alpha_{k'}} \E[f({x}_{k}) - f(x)] &\leq \frac{\frac1{2} \|x_0 - x\|^2+ \sum_{k=0}^{K-1} \alpha_k^2 ( (G^2 + \sigma^2) + \frac{\eta_k}{\alpha_k}) }{\sum_{{k'}=0}^{K-1} \alpha_{k'}} \nonumber \\
	\E\bigg[ f\bigg(\sum_{k=0}^{K-1} \frac{\alpha_k \cdot {x}_{k}}{\sum_{{k'}=0}^{K-1} \alpha_{k'}}  \bigg)\bigg] - f(x) &\leq
	\end{align}
	Next if we choose $\eta_k = \alpha_k (G^2 + \sigma^2)$, and set $x = x^* \in \argmin_{x' \in \cX} f(x')$ and $\alpha_k = \frac{D_\cX}{2\sqrt{G^2 + \sigma^2}\sqrt{K}}$ we get
	\begin{align}
	\E\bigg[ f\bigg(\sum_{k=0}^{K-1} \frac{\alpha_k \cdot {x}_{k}}{\sum_{{k'}=0}^{K-1} \alpha_{k'}}  \bigg)\bigg] - f(x^*) &\leq \frac{\frac1{2} D_\cX^2 + \sum_{k=0}^{K-1} \alpha_k^2 \cdot 2 (G^2 + \sigma^2) }{\sum_{{k'}=0}^{K-1} \alpha_{k'}} \nonumber \\
	&= \frac{2 \sqrt{G^2 + \sigma^2} D_\cX}{\sqrt{K}}
	\end{align}
	Clearly the algorithm uses $K$ SFO calls.
	At step $k$ when approximating the projection using an LMO based method, after using $\hat{T}_k = \lceil \frac{7 D_\cX^2}{\alpha_k^2 (G^2 + \sigma^2)} \rceil$ LMO calls in the \projstep procedure, the Wolfe duality gap \eqref{eq:envelope-subgrad-FWdual} is at most $\ceil{\frac{6 (1/\alpha_k) D_{\cX}^2}{\hat{T}_k}} \leq \alpha_k(G^2 + \sigma^2)$ if we use CndG procedure~\cite[Theorem 2.2(c)]{lan2016conditional} or $\ceil{\frac{7 (1/\alpha_k) D_{\cX}^2}{\hat{T}_k}} \leq \alpha_k(G^2 + \sigma^2)$ if we use the standard Frank-Wolfe algorithm~\cite[Theorem 2]{jaggi2013revisiting}. Therefore the total number of linear minimization oracle calls made by the algorithm is
	\begin{align}
	\sum_{k=0}^{K-1} \hat{T}_k = \sum_{k=0}^{K-1} \frac{7 D_\cX^2}{\alpha_k^2 (G^2 + \sigma^2)} + K = 28 K^2 + K = O(K^2)
	\end{align}
	where we use the given choice for $\alpha_k = \frac{D_\cX}{2\sqrt{G^2 + \sigma^2}\sqrt{K}}$.
\end{proof}

\section{Proofs of the main results}
\label{sec:proof} 

\subsection{Proof of Lemma \ref{lem:joint-gap-to-func-gap}}\label{sec:joint-gap-to-func-gap-pf}
\begin{proof}

First we prove part $(i)$. By definitions of $\Psi_\lambda$ \eqref{eq:joint_moreau_opt} and $f_\lambda$ (Definition~\ref{def:Moreau}), we have $\min_{x \in\cX} \min_{x \in \bX} \Psi_\lambda(x,x') = \min_{x \in \cX} f_\lambda(x)$. By Lemma~\ref{lem:moreau-properties}(a), we also can show that $\min_{x \in \cX} f_\lambda(x) \leq \min_{x \in \cX} f(x)$.

For part $(ii)$, first we show the following.
\begin{align}
\E f_\lambda (x_\varepsilon) = \E \Psi_\lambda (x_\varepsilon, \hx_\lambda(x_\varepsilon)) = \E \min_{x'(x_\varepsilon)} \Psi_\lambda (x_\varepsilon, x'(x_\varepsilon)) &\leq \E_{\bar{x}_K} \Psi_\lambda (x_\varepsilon, \E_{x_\varepsilon'|x_\varepsilon} x_\varepsilon') \nonumber \\
&\leq \E_{x_\varepsilon} \E_{x_\varepsilon'|x_\varepsilon} \Psi_\lambda (x_\varepsilon, x_\varepsilon') \nonumber \\
&= \E \Psi_\lambda (x_\varepsilon, x_\varepsilon')
\end{align}
Finally, combining the above inequality with Lemma~\ref{lem:moreau-properties}(c) we get the desired result
\begin{align}
\E f(x_\varepsilon) &\leq \E f_\lambda (x_\varepsilon) + G^2\lambda/2 \leq \E \Psi_\lambda (x_\varepsilon, x_\varepsilon') + G^2 \lambda/2
\end{align}
\end{proof}

\subsection{Analysis of \aproj~(Algorithm~\ref{algo:envelope_subgrad_method_proj}) and \almo~(Algorithm~\ref{algo:envelope_subgrad_method_lmo}) method}
\label{sec:envelope-subgrad-method-thm-pf}
Instead of separately analyzing \aproj~and \almo, we first analyze a more general algorithm, Algorithm~\ref{algo:envelope_subgrad_method}, which has the following guarantee.
\begin{algorithm}[t]
	\caption{Moreau subgradient method for nonsmooth convex optimization using PO or LMO}
	\label{algo:envelope_subgrad_method}
	\DontPrintSemicolon %
	\SetKwFunction{PROJSTEP}{Approx-Proj}
	\SetKwFunction{PROXSLIDE}{Prox-Slide}	
	\KwIn{$f$, 
		$\cX$, $\bX$,
		$G$, $D_{\cX}$, $R$, $x_0$, $K$, 
		$\tilde{D}$, $\lambda$, 
		$\{\eta_k \in \reals_{+}\}_{k\in[K]}$ %
	}
	
	\SetKwProg{Fn}{}{:}{}
	{
		Set ${\px}_0 ' = {\dx}_0 ' = {\px}_0 = {\dx}_0 = x_0$\;
		\For{$k = 1, \ldots, K$} {
			Set $\lambda_k = \lambda$, $\beta_{k} = \frac{4}{\lambda_k k}\;, \;\gamma_{k} = \frac{2}{k+1}\,, \text { and } 
			T_{k} = \Big \lceil{ \frac{(4G^2 + \sigma^2) \lambda^2 K k^2}{2\tilde{D}}}\Big\rceil$
			\nllabel{algo_line:envelope-subgrad-param-set}\;
			
			Set $({\mx}_{k}, {\mx}_{k}') =\left(1-\gamma_{k}\right) \cdot ({\px}_{k-1},{\px}_{k-1}')+\gamma_{k} \cdot ({\dx}_{k-1}, {\dx}_{k-1}') $
			\nllabel{algo_line:midpoint_compute}\;
			
			Set ${\dx}_{k} =\PROJSTEP\left(\nabla_{\mx_k} \Psi_{\lambda}(\mx_k,\mx_k'), {\dx}_{k-1}, \beta_k, \eta_k\right)$ \tcp{\color{violet} Note $\nabla_{\mx_k} \Psi_{\lambda}(\mx_k,\mx_k') = \frac{\mx_k - \mx_k'}{\lambda}$}
			\nllabel{algo_line:dual_compute}
			
			Set $\left({\dx}_{k}', {\tdx}_{k}'\right)=\text{\PROXSLIDE}\big(\nabla_{\mx_k'} \psi_{\lambda}(\mx_k,\mx_k'), {\dx}_{k-1}', \beta_{k}, T_{k}\big)$ 
			\tcp*{\color{violet}  $\nabla_{\mx_k'} \psi_{\lambda}(\mx_k,\mx_k')=\frac{{\mx}_{k}' - {\mx}_{k}}\lambda$}
			\nllabel{algo_line:dual_compute_prime}
			
			Set $({\px}_{k}, {\px}_{k}') =\left(1-\gamma_{k}\right) \cdot ({\px}_{k-1},{\px}_{k-1}') +\gamma_{k} \cdot ({\dx}_{k}, {\tdx}_{k}')$
			\nllabel{algo_line:primal_compute}\;
		}
		\KwOut{$({\px}_{K}, {\px}'_{K})$}
	}
	
	\vspace*{5pt}
	\Fn{\PROJSTEP{$g$, 
			$u_0$,
			$\beta$, $\eta$}{\color{violet}\ // Approx. resolve {\small $\cP_\cX\big(u_0 -g/\beta \big)$}\cite{lan2016conditional}}\nllabel{algo_line:projstep}
		}{
		Either using exact PO, $\mathcal{P}_\cX$ \eqref{eq:nsco} , or using any competitive LMO based algorithm (e.g.~Frank-Wolfe method~\cite{frank1956algorithm} or ${\rm CndG}$ procedure~\cite[Algo.~1]{lan2016conditional}), approximately solve the projection problem
		\begin{align}		
		u_\Pi \approx \argmin _{u \in \cX} \Ip{g}{u} +\frac\beta2 \|u-u_0\|^2 = \argmin _{u \in \cX} \frac\beta2 \|u-(u_0-g/\beta)\|^2\,,\label{eq:envelope-subgrad-FWproj}
		\end{align}
		ensuring that the Wolfe duality gap~\cite{jaggi2013revisiting} of the above problem at $u_{\Pi}$ satisfies
		\begin{align}\label{eq:envelope-subgrad-FWdual}
		\max _{ s \in \cX} \Ip{g +\beta \left(u_\Pi - {u_0}\right)}{{u}_{\Pi} - s } \leq \eta_k\,
		\end{align}
		\KwRet{$u_\Pi$}		
	}

	\vspace*{5pt}
	\Fn{\PROXSLIDE{$g$, 
			$u_0$,
			$\beta$, $T$}{\color{violet}\ // Approx. resolve {\small ${\prox}_{f/\beta}\big(u_0 -g/\beta \big)$}\cite{lan2016gradient}}\nllabel{algo_line:proxslide}
	}{
		Set $\widetilde{u}_{0} = u_{0}$\;
		\For{$t = 1, \ldots, {T}$} {
			Set $\theta_{t} = \frac{2(t+1)}{t(t+3)}$, 
			$\ \ \hg_{t-1} = \sfo{u_{t-1}}$ \eqref{eq:sfo}
			\nllabel{algo_line:sfo} \nllabel{algo_line:prox-slide-param-set}\;
			Set 
			$\widehat{u}_{t} = u_{t-1} - \frac{1}{(1+t/2)\beta} \cdot (\hg_{t-1} + \beta (u_{t-1} -(u_{0} - g/\beta)))$
			\linebreak \tcp*{\color{violet}
				subgradient method step for $\phi(u) \defeq f(u) + \frac\beta2 \|u-\big(u_{0} - \frac{g}{\beta} \big)\|^2$}
			\nllabel{algo_line:sgd_step}
			Set $u_t = \widehat{u}_t \cdot \min\left(1,R/\norm{u_t}\right)$\tcp*{\color{violet} projection of $\widehat{u}_t$ onto $\bX$: $\mathcal{P}_\bX(\u_t)$}
			Set $\widetilde{u}_{t}=\big(1-\theta_{t}\big) \cdot \widetilde{u}_{t-1}+\theta_{t} \cdot u_{t}$
			\nllabel{algo_line:prox-slide-avg}\;
		}
		\KwRet{$(u_{T}, \widetilde{u}_{T})$}\;
	}
\end{algorithm}

\begin{theorem}\label{thm:envelope-subgrad-method-thm}
	Let $f:\bX \to \reals$ be a $G$-Lipschitz continuous proper l.s.c.~convex function, and
	$\cX\subseteq \bX = B(0, R)$ be some convex subset contained inside the Euclidean ball of radius $R$ around origin. Then after $K$ iterations, the Algorithm~\ref{algo:envelope_subgrad_method} 
	outputs ${\px}_K \in \cX$ satisfying 
	\begin{align}
	\E[f\left({\px}_{K}\right)] - f(x^*) \leq &\frac{10 \|x_0 - x^*\|^2 + 8 \tilde{D}}{\lambda  K(K+1)} + \frac{\sum_{k=1}^{K} 2k\, \eta_k }{K(K+1)} + G^2 \frac\lambda2
	\end{align}
	 for any choice of $\lambda > 0$, $\tilde{D} > 0$, and tolerance $\{\eta_k\}_{k\in[K]}$ \eqref{eq:envelope-subgrad-FWdual}.  
	
\end{theorem}
{\bf Remarks}:
Before providing a proof for the above result we discuss some its implications. \aproj makes $K$ PO calls, one per outer step, and 
$\sum_{k=1}^K T_k = \Ord(\lambda^2 K^4) $  SFO calls, one per inner step.  
The above analysis shows that we need to choose $\lambda=\varepsilon/G^2$, which is expected from Lemma~\ref{lem:moreau-properties}. 
Since PO returns exact projections, the second term is zero with $\eta_k=0$. 
The target accuracy of $\varepsilon$ is achieved by tuning the first term, where we need to choose $K=\Theta(1/\sqrt{\lambda \varepsilon})$. 
Put together, this gives the desired $\Ord(\varepsilon^{-1})$ 
\pocc
and $\Ord(\varepsilon^{-2})$ 
\sfocc
for \aproj.
A complete proof is provided in Section~\ref{sec:envelope-subgrad-method-cor1-pf}. 

When we have inexact projections in \almo, 
we need $\eta_k=\Theta(1/k)$ to ensure that the second term is $\Ord(\varepsilon)$. 
At (outer) iteration $k$, this uses $\hat{T}=\Omega(K)$ iterations of Frank-Wolfe algorithm in \FWProj of Algorithm~\ref{algo:envelope_subgrad_method_lmo}.  
\almo makes $\sum_{k=1}^K \Ord(K) = \Ord(K^2)$ LMO calls, resulting in $\Ord(\varepsilon^{-2})$
\lmocc
as $K=\Ord(1/\sqrt{\lambda \varepsilon}) = \Ord(1/\varepsilon)$. 
A complete proof is in Section~\ref{sec:envelope-subgrad-method-cor2-pf}. 

\begin{proof}[Proof of Theorem~\ref{thm:envelope-subgrad-method-thm}]
Now consider the following potential (Lyapunov) function: 
\begin{align}\label{eq:almo-potential-func}
\Phi_k \defeq k(k+1) (\Psi_\lambda({\px}_{k},{\px}_{k}') - \Psi_\lambda(x,x')) + \frac{4}{\lambda} (\| {\dx}_{k} - x\|^2 + \frac{(T_{k+1}+1)(T_{k+1}+2)}{T_{k+1}(T_{k+1}+3)} \|  {\dx}_{k}' - x'\|^2)
\end{align}
This is a slightly modified version of the following potential function for the standard AGD setting with a $2/\lambda$-smooth function $\Psi_\lambda$ \cite{bansal2017potential}: $k(k+1) (\Psi_\lambda({\px}_{k},{\px}_{k}') - \Psi_\lambda(x,x')) + \frac{4}{\lambda} (\| {\dx}_{k} - x\|^2 + \|  {\dx}_{k}' - x'\|^2)$.
 Notice that the modification factor
\begin{align}
\frac{(T_{k+1}+1)(T_{k+1}+2)}{T_{k+1}(T_{k+1}+3)} \leq \frac32 = \Ord(1)
\end{align}
is upper-bounded by a constant when  $1 \leq T_k$. Below we prove that this potential satisfies the approximate descent guarantee: $\Phi_{k} \leq \Phi_{k-1} + k \eta_k + k \eta_k'$, for some error $\eta_k'$. First, notice that by $2/\lambda$-smoothness and convexity of $\psi_\lambda$
\begin{align} \label{eq:general-thm-pre-descent}
\psi_\lambda({\px}_{k},{\px}_{k}') \leq\; &\psi_\lambda({\mx}_{k},{\mx}_{k}') + \Ip{\nabla_k }{({\px}_{k},{\px}_{k}') - ({\mx}_{k},{\mx}_{k}')} + \frac1{\lambda}  \|({\px}_{k},{\px}_{k}') - ({\mx}_{k},{\mx}_{k}')\|^2 \nonumber \\
\overset{}=\; & (1-\gamma_k) [\psi_\lambda({\mx}_{k},{\mx}_{k}') + \Ip{\nabla_k }{({\px}_{k-1},{\px}_{k-1}') - ({\mx}_{k},{\mx}_{k}')}] \nonumber \\
&\gamma_k [\psi_\lambda({\mx}_{k},{\mx}_{k}') + \Ip{\nabla_k }{({\dx}_{k},{\tdx}_{k}') - ({\mx}_{k},{\mx}_{k}')} + \frac{\gamma_k}{\lambda}  \|({\dx}_{k},{\tdx}_{k}') - ({\dx}_{k-1},{\dx}_{k-1}')\|^2] \nonumber \\
\leq\; &(1-\gamma_k)\psi_\lambda({\px}_{k-1},{\px}_{k-1}') + \nonumber \\
&\gamma_k [\psi_\lambda({\mx}_{k},{\mx}_{k}') + \Ip{\nabla_k }{({\dx}_{k},{\tdx}_{k}') - ({\mx}_{k},{\mx}_{k}')} + \frac{\gamma_k}{\lambda}  \|({\dx}_{k},{\tdx}_{k}') - ({\dx}_{k-1},{\dx}_{k-1}')\|^2]
\end{align}
where we use the shorthand $\nabla_k \defeq [\nabla_{k,x}^T \nabla_{k,x'}^T]^T \defeq [\nabla_{x} \psi_\lambda({\mx}_{k},{\mx}_{k}')^T \; \nabla_{x'} \psi_\lambda({\mx}_{k},{\mx}_{k}')^T]^T$, and second inequality used \Cref{algo_line:midpoint_compute,algo_line:primal_compute}. 
Now combining this with $f({\px}_{k}') \leq (1-\gamma_k) f({\px}_{k-1}') +  \gamma_k f({\tdx}_{k}')$ (using convexity of $f$ and \Cref{algo_line:primal_compute}) and $\gamma_k/\lambda = 2/(\lambda (k+1)) \leq {2}/{(\lambda \, k)} = \beta_k/2$ (using \Cref{algo_line:envelope-subgrad-param-set}), and multiplying it with $k(k+1)$ we get that
\begin{align}
&k(k+1)\Psi_\lambda({\px}_{k},{\px}_{k}') \nonumber \\
\leq\; &k(k-1) \Psi_\lambda({\px}_{k-1},{\px}_{k-1}') + 2k \psi_\lambda({\mx}_{k},{\mx}_{k}') +
2k [ \Ip{\nabla_{k,x}}{{\dx}_{k} - {\mx}_{k}} + \frac{\beta_k}{2}  \|{\dx}_{k} - {\dx}_{k-1}\|^2] \nonumber \\
&2k [ f({\tdx}_{k}') + \Ip{\nabla_{k,x'}}{{\tdx}_{k}' - {\mx}_{k}'} + \frac{\beta_k}{2}  \|{\tdx}_{k}' - {\dx}_{k-1}'\|^2] \nonumber \\
=\; &k(k-1) \Psi_\lambda({\px}_{k-1},{\px}_{k-1}') + 2k \psi_\lambda({\mx}_{k},{\mx}_{k}') +
2k [ \Ip{\nabla_{k,x}}{{\dx}_{k} - {\mx}_{k}} + \frac{\beta_k}{2}  \|{\dx}_{k} - {\dx}_{k-1}\|^2] \nonumber \\
&2k [ \phi_k({\tdx}_{k}') - \phi_k(x') + \frac{\beta_k}{2}  \|x' - {\dx}_{k-1}'\|^2 + f(x') +\Ip{\nabla_{k,x'}}{x' - {\mx}_{k}'} ]
\label{eq:general-thm-eq5b}
\end{align}
where for brevity we use the notation
\begin{align}\label{eq:general-thm-inner-opt}
\phi_k(x') \defeq f\left(x'\right)  +
\Ip{\nabla_{k,x'} }{x'} + \frac{\beta_k}{2}\left\|x'-{\dx}_{k-1}'\right\|^{2}\,.
\end{align}
Now using the approximate optimality of ${\dx}_{k}$ through the bound on the Wolfe dual gap~\eqref{eq:envelope-subgrad-FWdual} we get,
\begin{align}
\frac{\beta_k}{2}\left\|{\dx}_{k} - {\dx}_{k-1}\right\|^{2} &= \frac{\beta_k}{2}\left\|{\dx}_{k-1} - {x}\right\|^{2} + \beta_k \Ip{{\dx}_{k} -{\dx}_{k-1}}{{\dx}_{k}-x} - \frac{\beta_k}{2}\left\|{\dx}_{k} - {x}\right\|^{2} \nonumber \\
&\leq \frac{\beta_k}{2}\left\|{\dx}_{k-1} - {x}\right\|^{2} + \Ip{\nabla_{k,x} }{x -{\dx}_{k}} + \eta_k - \frac{\beta_k}{2}\left\|{\dx}_{k} - {x}\right\|^{2}\,.\label{eq:general-thm-eq6}
\end{align}
When ${\dx}_{k}$ is the exact projection (as in \Cref{algo_line:dual_compute_proj}) then the above inequality is satisfied with above $\eta_k = 0$. Otherwise, with the LMO oracle we will later set $\eta_k = \Ord(\varepsilon)$. Next we state the following lemma which provides a guarantee for the \proxslide procedure~\cite{lan2016gradient}. Here for rigorousness, we denote the iterates of the \proxslide procedure (\Cref{algo_line:proxslide}) called at the outer step $k$, with $\{u_{k,t}\}_t$. Similarly at the outer step $k$, we denote the stochastic subgradients used by the \proxslide procedure and the corresponding subgradient with $\{\hg_{k,t}\}_t$ and $\{g_{k,t}\}_t$, i.e.~$g_{k,t} \defeq \E[\hg_{k,t} | u_{k,t}] \in \partial f(u_{k,t})$ for all $k$ and $t$. A proof for this lemma is provided in Section~\ref{prop:general-thm-prop-2-pf}. 
\begin{proposition}[{\cite[Similar to Proposition 1]{lan2016gradient}}]\label{prop:general-thm-prop-2}
Let $\phi_k$ \eqref{eq:general-thm-inner-opt} be the minimization objective solved by \proxslide procedure at step $k$ of Algorithm~\ref{algo:envelope_subgrad_method}. Then $({\dx}_{k}', {\tdx}_k')$ obtained after $T_k$ iterations of the procedure satisfy the following for any $x' \in \bX$,
\begin{align}
\phi_k({\tdx}_k') - \phi_k (x') &\leq \frac2{T_k(T_k+3)} \frac{\beta_k}2 \|{\dx}_{k-1}'-x'\|^2 - \frac{(T_k+1)(T_k+2)}{T_k(T_k+3)} \frac{\beta_k}2 \|{\dx}_{k}'-x'\|^2 + \nonumber \\ 
&\;\;\;\;\; \frac{4\,\sum_{t=0}^{T_k-1} (2G + \|\delta_{k,t}\|)^2}{\beta_k T_k(T_k+3) }  + \sum_{t=0}^{T_k-1} \frac{2(t+2)}{T_k(T_k+3)} \Ip{\delta_{k,t}}{x' - u_{k,t}}
\end{align}
where $\delta_{k,t} \defeq \hg_{k,t} - g_{k,t}$ and $u_{k,t}$ are private inner variable of the \proxslide procedure.
\end{proposition}
{\color{violet} {\it Aside:} Note that \proxslide procedure essentially applies $T_k$ steps of the proximal standard subgradient method to the $\phi_k$ \eqref{eq:general-thm-inner-opt}, which is a composite function of a $G$-Lipschitz function $f$ and \prox-friendly $\beta_k$-strongly convex quadratic. Finally the procedure outputs the average of its iterate ${\tdx}_k'$ and its last iterate ${\dx}_k'$. In the end we will set $T_k = \Theta(1/\varepsilon)$ and $K = \Theta(1/\varepsilon)$ so that total number of subgradients used by the algorithm be $\sum_{k=1}^K T_k = \Ord(1/\varepsilon^2)$.}

Now substituting \eqref{eq:general-thm-eq6} and Proposition~\ref{prop:general-thm-prop-2} into \eqref{eq:general-thm-eq5b} we get
\begin{align}
k(k+1)\Psi_\lambda({\px}_{k},{\px}_{k}') 
\leq\; &k(k-1) \Psi_\lambda({\px}_{k-1},{\px}_{k-1}') + 2k \psi_\lambda({\mx}_{k},{\mx}_{k}') +
\nonumber \\&
2k [ \Ip{\nabla_{k,x}}{x - {\mx}_{k}} + \frac{\beta_k}{2}  (\|{\dx}_{k-1} - x\|^2 - \|{\dx}_{k} - x\|^2) + \eta_k] \nonumber \\
&2k [ f(x') + \Ip{\nabla_{k,x'}}{x' - {\mx}_{k}'} ] + \nonumber \\
&2k [\frac{(T_k+1)(T_k+2)}{T_k(T_k+3)} \frac{\beta_k}2 \|{\dx}_{k-1}'-x'\|^2 - \frac{(T_k+1)(T_k+2)}{T_k(T_k+3)} \frac{\beta_k}2 \|{\dx}_{k}'-u\|^2 + \eta_k'] \nonumber \\ 
\leq\; &k(k-1) \Psi_\lambda({\px}_{k-1},{\px}_{k-1}') + 2k \Psi_\lambda(x,x') + 2k (\eta_k + \eta_k')
\nonumber \\&
\frac4\lambda ( \|{\dx}_{k-1} - x\|^2 - \|{\dx}_{k} - x\|^2) \nonumber \\
&\frac4\lambda ( \frac{(T_k+1)(T_k+2)}{T_k(T_k+3)} \|{\dx}_{k-1}'-x'\|^2 - \frac{(T_{k+1}+1)(T_{k+1}+2)}{T_{k+1}(T_{k+1}+3)} \|{\dx}_{k}'-x'\|^2 ) \,,
\end{align}
where we use the shorthand 
\begin{align}
\eta_k' \defeq \frac{4\,\sum_{t=0}^{T_k-1} (2G + \|\delta_{k,t}\|)^2}{\beta_k T_k(T_k+3) }  + \sum_{t=0}^{T_k-1} \frac{2(t+2)}{T_k(T_k+3)} \Ip{\delta_{k,t}}{x' - u_{k,t}}\,,\label{eq:proxslide_error}
\end{align}
and the last inequality uses convexity of $\psi_\lambda$ and $\Psi_\lambda = f + \psi_\lambda$, definition of $\beta_k$ (\Cref{algo_line:envelope-subgrad-param-set}), and the fact that
\begin{align}
T_{k} \leq T_{k+1} \quad \text{(\Cref{algo_line:envelope-subgrad-param-set})} \quad \text{, and } \quad \frac{(T_{k+1}+1)(T_{k+1}+2)}{T_{k+1}(T_{k+1}+3)} \leq \frac{(T_{k}+1)(T_{k}+2)}{T_{k}(T_{k}+3)} 
\end{align}
This proves the approximate descent guarantee: $\Phi_{k} \leq \Phi_{k-1} + k (\eta_k + \eta_k')$, which along with the facts: $1 \leq T_1$ and ${\dx}_0 = {\dx}_0' = x_0$ gives
\begin{align}
\Psi_\lambda({\px}_{K},{\px}_{K}') - \Psi_\lambda(x,x') \leq &\frac{4 (\|x_0 - x\|^2 + (3/2) \|x_0-x'\|^2)}{\lambda  K(K+1)} +
\frac{\sum_{k=1}^{K} 2k (\eta_k + \eta_k')}{K(K+1)} \label{eq:envelope-subgrad-method-thm-eq1}
\end{align}
Now we take expectation, with respect to randomness in all the stochastic subgradients $((\hg_{k,i})_{i=1}^{T_k})_{k=1}^K$ used in the algorithm, on both sides of \eqref{eq:envelope-subgrad-method-thm-eq1}. Then the expectation of the error from the \proxslide procedure can be bounded as follows
\begin{align}
\sum_{k=1}^{K} 2k \E[\eta_k'] &= \sum_{k=1}^{K} 2k \E[\frac{4\,\sum_{t=0}^{T_k-1} (2G + \|\delta_{k,t}\|)^2}{\beta_k T_k(T_k+3) }  + \sum_{t=0}^{T_k-1} \frac{2(t+2)}{T_k(T_k+3)} \Ip{\delta_{k,t}}{u - u_{t}}] \nonumber \\
&\leq \sum_{k=1}^{K} 2k \frac{8\, (4G^2 + \sigma^2)}{(\frac4{\lambda k}) (\frac{(4 G^2 + \sigma^2) \lambda^2 K k^2}{2\tilde{D}}) }  + 0 \nonumber \\
&= \frac{8 \tilde{D}}{\lambda} \label{eq:envelope-subgrad-method-thm-eq2}
\end{align}
where we use \eqref{eq:proxslide_error}, linearity of expectation, $(a+b)^2 \leq 2(a^2 + b^2)$, variance of stochastic gradient $\E[\|\delta_{k,t}\|^2 | u_{k,t}] = \E[\|\hg_{k,t} - g_{k,t}\|^2 | u_{k,t}] \leq \sigma^2$ \eqref{eq:sfo}, the value of $T_k$ from \Cref{algo_line:envelope-subgrad-param-set}, and the fact that expectation of the second term becomes zero, since $\E [\hg_{k, i-1} \,|\, u_{k, i-1}] = g_{k, i-1}$, which in turn implies
\begin{align}
\E [\Ip{\delta_{k,t}}{x' - u_{k,t}}] &= \E  \big[\,\E[ \Ip{\hg_{k, t} - g_{k, t}}{x' - u_{k,t}} \,|\, u_{k,t}] \,\big] s= \E [ \Ip{0}{x' - u_{k,i-1}} ] = 0\;.  \label{eq:general-thm-eq11}
\end{align}
{\color{violet} {\it Aside:} Note that, in the final guarantee, when we set $\lambda = \varepsilon/G^2$ and $K = \Ord(1/\varepsilon)$, we are setting $T_k = \Theta(\varepsilon\, k^2) = \Ord(1/\varepsilon)$ and $1/\beta_k = \Theta(1/\varepsilon\,k) = \Ord(1)$, so that the error from the \proxslide procedure is small enough. For example, at $k=K$, $\E[\eta_K] = \Ord((G^2 + \sigma^2)/\beta_K T_K) = \Ord(\varepsilon)$.}

Now taking expectation on both sides of \eqref{eq:envelope-subgrad-method-thm-eq1} and using linearity of expectation and \eqref{eq:envelope-subgrad-method-thm-eq2} we get that
\begin{align}
\E[\Psi_\lambda \left({\px}_{K}, {\px}_{K}'\right)] - \Psi_\lambda(x,x') \leq &\frac{4 (\|x_0 - x\|^2 + (3/2) \|x_0-x'\|^2 + 2\tilde{D})}{\lambda  K(K+1)} +
\frac{\sum_{k=1}^{K} 2k\, \eta_k }{K(K+1)} \label{eq:envelope-subgrad-method-thm-eq3}
\end{align}
Next setting $x'=x = x^* \in \cX \subseteq \bX$ and using Lemma~\ref{lem:joint-gap-to-func-gap} and \eqref{eq:joint_moreau_opt} we get that
\begin{align}
\E[f\left({\px}_{K}\right)] - f(x^*) \leq \frac{10\|x_0 - x^*\|^2 + 8\tilde{D}}{\lambda  K(K+1)} + \frac{\sum_{k=1}^{K} 2k\, \eta_k }{K(K+1)} + G^2 \frac\lambda2 \label{eq:envelope-subgrad-method-thm-eq5}
\end{align}
{\color{violet} {\it Aside:} Note that, in the final guarantee, when we set $\lambda = \varepsilon/G^2$, the third term, which is the error from the Moreau smoothing becomes $\varepsilon/2$. Additionally, when $K = \Ord(1/\varepsilon)$, first term above is $\Ord(\varepsilon)$. Further, when we set $T_k = \Theta(\varepsilon\,k^2) = \Ord(1/\varepsilon)$, we get $1/\beta_k = \Ord(1/\varepsilon\,k)$ and $\E[\eta_k] = \Ord(1/k)$ so that the second term is also $\Ord(\varepsilon)$.}
\end{proof}

Next using the above result we derive the guarantees for \aproj~(Algorithm~\ref{algo:envelope_subgrad_method_proj}) and \almo~(Algorithm~\ref{algo:envelope_subgrad_method_lmo}) as corollaries of Theorem~\ref{thm:envelope-subgrad-method-thm}.

\subsubsection{Proof of Theorem \ref{thm:envelope-subgrad-method-cor1}}

\label{sec:envelope-subgrad-method-cor1-pf}

\begin{proof}
Notice that exact projection on \Cref{algo_line:dual_compute_proj} of Algorithm~\ref{algo:envelope_subgrad_method_proj} is equivalent to choosing $\eta_k = 0$ in Algorithm~\ref{algo:envelope_subgrad_method}. Then setting, $\eta_k = 0$, and $\lambda = \varepsilon/G^2$ in Theorem~\ref{thm:envelope-subgrad-method-thm} we get
\begin{align}\label{eq:envelope-subgrad-method_proj-eq1}
\E[f\left({\px}_{K}\right)] - f(x^*) \leq \frac{G^2 (10\|x_0 - x^*\|^2 + 8\tilde{D})}{\varepsilon  K(K+1)} + 0 + \frac\varepsilon2
\end{align}
Now, by using the given choices: $\tilde{D} = c \|x_0 - x^*\|^2$ and $K = \lceil \frac{2 \sqrt{10+8c} \, G \|x_0 - x^*\|}{\varepsilon}\rceil$, we get
\begin{align}\label{eq:envelope-subgrad-method_proj-eq2}
\E[f\left({\px}_{K}\right)] - f(x^*) 
&\leq \varepsilon
\end{align}
Then the number of PO calls made by the algorithm is $K = \Ord({\frac{G\|x_0 - x^*\|}{\varepsilon}})$ and the total number of SFO calls made subgradients made is
\begin{align}
\sum_{k=1}^{K} T_{k} \leq \sum_{k=1}^{K}\left(\frac{(4G^{2} + \sigma^2) \lambda^2 K k^{2}}{2 \tilde{D}}+1\right) &= \frac{(4G^{2} + \sigma^2) \varepsilon^2 K^2(K+1)(2K+1)}{12 c G^{4} \| x_{0} - x^*_{\lambda} \|^2}+ K \nonumber \\
&= \order{\frac{(G^{2} + \sigma^2) \|x_0 - x^*_{} \|^2}{\varepsilon^2}}\,,
\end{align}
where we used \Cref{algo_line:envelope-subgrad-param-set_proj} and the given choices for $\lambda$, $K$, and $\tilde{D}$.
\end{proof}

\subsubsection{Proof of Theorem \ref{thm:envelope-subgrad-method-cor2}}\label{sec:envelope-subgrad-method-cor2-pf}
\begin{proof}

Notice that at step $k$ of Algorithm~\ref{algo:envelope_subgrad_method_lmo} choosing $\hat{T}=\lceil \frac{7 K D_\cX^2}{c' \tilde{D}} \rceil  = \Ord(\frac1\varepsilon)$ is equivalent to choosing $\eta_k = \frac{4 c' \tilde{D}}{\lambda K k}$ in Algorithm~\ref{algo:envelope_subgrad_method} (see below). Therefore by setting, $\eta_k = \frac{4 c'\tilde{D}}{\lambda K k} = \Ord(\frac1k)$, and $\lambda = \varepsilon/G^2$ in Theorem~\ref{thm:envelope-subgrad-method-thm} we get
\begin{align}\label{eq:envelope-subgrad-method_lmo-eq1}
\E[f\left({\px}_{K}\right)] - f(x^*) \leq &\frac{G^2 (10\|x_0 - x^*\|^2 + 8\tilde{D} + 8c' \tilde{D})}{\varepsilon K(K+1)} + \frac\varepsilon2
\end{align}
Now, by using the given choices: $\tilde{D} = c \|x_0 - x^*\|^2$ and $K = \lceil \frac{2 \sqrt{10+8c(1+c')} G \|x_0 - x^*\|}{\varepsilon}\rceil$, in the \eqref{eq:envelope-subgrad-method_lmo-eq1} we get
\begin{align}\label{eq:envelope-subgrad-method_lmo-eq2}
\E[f\left({\px}_{K}\right)] - f(x^*) %
&\leq \varepsilon
\end{align}
Then using the similar arguments as in proof of Theorem~\ref{thm:envelope-subgrad-method-cor2}, we can show that $K = \Ord({\frac{G\|x_0 - x^*\|}{\varepsilon}})$ and the total number of SFO calls made is $\sum_{k=1}^{K} T_{k} = \Ord({\frac{(G^{2} + \sigma^2) \|x_0 - x^*_{} \|^2}{\varepsilon^2}})$.

Finally we calculate the total number of LMO calls made. 
At outer step $k$ of Algorithm~\ref{algo:envelope_subgrad_method_lmo}, after using $\hat{T} = \lceil \frac{7 K D_\cX^2}{c' \tilde{D}} \rceil$ LMO calls in the \projstep procedure, the Wolfe duality gap \eqref{eq:envelope-subgrad-FWdual} is at most $\ceil{\frac{6 \beta_k D_{\cX}^2}{\hat{T}}} \leq \frac{4 c'\tilde{D}}{\lambda K k} = \eta_k$ if we use CndG procedure~\cite[Theorem 2.2(c)]{lan2016conditional} or $\ceil{\frac{7 \beta_k D_{\cX}^2}{\hat{T}}} \leq \frac{4 c'\tilde{D}}{\lambda K k} = \eta_k$ if we use the standard Frank-Wolfe algorithm~\cite[Theorem 2]{jaggi2013revisiting}. Therefore the total number of linear minimization oracle calls made by the algorithm is
\begin{align}
K \hat{T} = \frac{7 K^2 D_{\cX}^2}{ cc'\|x_0 - x^*\|^2} + K = \Ord \Big( \frac{G^2 D_{\cX}^2}{\varepsilon^2}\Big)\,,
\end{align}
where we used \Cref{algo_line:envelope-subgrad-param-set_proj} and the given choices for $K$ and $\tilde{D}$.	
\end{proof}

\subsubsection{Proof of Proposition~\ref{prop:general-thm-prop-2}: Analysis of \proxslide (\Cref{algo_line:proxslide}) procedure}
\label{prop:general-thm-prop-2-pf}
\begin{proof}
	We analyze the \proxslide procedure for a fixed $k$, therefore we drop $k$ from $\phi_k$, $u_{k, t}$, $\hg_{k, t}$, and $\delta_{k, t}$, which are denoted here with $\phi$, $u_t$, $\hg_t$, and $\delta_t$. \proxslide has the following update steps.
	\begin{align}
	\theta_{t} &=  \frac{2(t+1)}{t(t+3)} \,,\;\; \hg_{t-1} = \sfo{u_{t-1}}\\
	\widehat{u}_{t} &= u_{t-1} - \frac{1}{(1+t/2)\beta} \cdot (\hg_{t-1} + \beta (u_{t-1} -(u' - g/\beta))) \\
	u_{t} &= \widehat{u}_{t} \cdot \min(1, 2R/\|\widehat{u}_{t}\|) \\
	\widetilde{u}_{t} &= \left(1-\theta_{t}\right) \widetilde{u}_{t-1}+\theta_{t} u_{t}
	\end{align}
	By convexity and $G$-Lipschitzness of $f$ in $\bX$, for any $u \in \bX$, we get
	\begin{align}
	f(u_{t+1}) - f(u) &=  f(u_{t+1}) - f(u_t) + f(u_{t}) - f(u) \nonumber \\
	&\leq \Ip{g_{t+1}}{u_{t+1} - u_t} + \Ip{g_{t}}{u_{t} - u}\nonumber \\
	&= \Ip{\hg_{t}}{u_{t+1} - u_t} + \Ip{g_{t+1} - g_t - \delta_t}{u_{t+1} - u_t} + \Ip{\hg_{t}}{u_{t} - u} - \Ip{\delta_t}{u_{t} - u}\nonumber \\
	&\leq \Ip{\hg_t}{u_{t+1} - u} + (2G + \|\delta_t\|)\,\|u_{t+1} - u_t\| + \Ip{\delta_t}{u - u_{t}}\,,
	 \label{eq:proxslide_eq1}
	\end{align}
	where we used the fact that $\delta_t = \hg_t - g_t$. Notice that $u_t = \widehat{u}_{t} \cdot \min(1, 2R/\|\widehat{u}_{t}\|)$ is the projection of $\widehat{u}_{t}$ onto $\bX = B(0, 2R)$. Therefore, using \Cref{algo_line:sgd_step}, we can re-write \proxslide update as
	\begin{align}
	u_{t+1} &= \argmin_{u \in \bX} \frac{(t+3)\beta}4 \; \|u-\widehat{u_t}\|^2 \nonumber \\
	&= \argmin_{u \in \bX} \frac{(t +3)\beta}4 \; \Big\|u - \Big(u_{t-1} - \frac1{(1+((t+1)/2)\beta}\cdot (\hg_{t-1} + \beta (u_{t-1} -(u_0 - g/\beta))) \Big) \Big\|^2 \nonumber \\
	&= \argmin_{u \in \bX} \Big[F_{t}(u) \defeq \Ip{g}{u} + \Ip{\hg_t}{u} + \frac{(t+1)\beta}4 \|u-u_t\|^2 + \frac\beta2 \|u-u_0\|^2\Big]
	\end{align}
	By $\beta(t+3)/2$-strong convexity of the quadratic update objective $F_{t}(u)$ and the optimality of $u_{t+1} \in \argmin_{u \in \bX} F_t(u)$, we get that for any $u \in \bX$
	\begin{align}
	\frac{\beta(t+3)}4 \|u_{t+1} - u\|^2 + F_{t}(u_{t+1}) &\leq F_{t}(u) \label{eq:proxslide_eq2}
	\end{align}
	We want to provide a lower bound on $\phi(u)$ \ref{eq:general-thm-inner-opt} which is defined as follows, when using the private notation of the \proxslide procedure by setting $u=x'$, $u' = \dx_{k-1}$, $g = \nabla_{k, x'}$, $\beta = \beta_k$.
	\begin{align}
	\phi(u) \defeq f\left(u\right)  +
	\Ip{g}{u} + \frac{\beta}{2}\left\|u-u_0\right\|^{2}\,.
	\end{align}
	Now adding together \eqref{eq:proxslide_eq1} and \eqref{eq:proxslide_eq2} and using the definitions of  $F_t$ and $\phi$ we get
	\begin{align}
	\phi(u_{t+1}) -\phi(u) &\leq \frac\beta2 (\frac{t+1}2 \|u_{t}-u\|^2 - \frac{t+3}2 \|u_{t+1}-u\|^2) + \nonumber \\&\;\;\;\;\;
	(2G + \|\delta_t\|)\,\|u_{t+1} - u_{t}\| - \frac{\beta \,(t+1)}{4}\|u_{t+1} - u_{t}\|^2  + \Ip{\delta_t}{u - u_{t}} \nonumber \\
	&\leq \frac\beta2 (\frac{t+1}2 \|u_{t}-u\|^2 - \frac{t+3}2 \|u_{t+1}-u\|^2)  +  \frac{(2G + \|\delta_t\|)^2}{\beta \, (t+1)} + \Ip{\delta_t}{u - u_{t}}
	\end{align}
	where the second inequality follows from $ax - b x^2/2 \leq a^2 /2b$. Now multiplying the above inequality is by $2(t+2)/T(T+3)$ and then summing over $t=\{0,\ldots,T-1\}$, we get
	\begin{align}
	\sum_{t=0}^{T-1} \frac{2(t+2)}{T(T+3)} (\phi(u_{t+1}) - \phi(u)) &\leq \frac\beta{2} \frac2{T(T+3)} (\|u_{0}-u\|^2 - \frac{(T+1)(T+2)}{2}\|u_{T}-u\|^2) + \nonumber \\ 
	&\;\;\;\;\; \frac{2}{T(T+3)} \frac{2\,\sum_{t=0}^{T-1} (2G + \|\delta_t\|)^2}{\beta}  + \sum_{t=0}^{T-1} \frac{2(t+2)}{T(T+3)} \Ip{\delta_t}{u - u_{t}} \nonumber \\
	\phi(\widetilde{u}_{T}) - \phi(u) &\leq
	\end{align}
	where the last inequality uses Jensen's inequality and $\widetilde{u}_T = \sum_{t=1}^T \frac{2(t+1)}{T(T+3)} u_t$, last of which follows from \Cref{algo_line:prox-slide-param-set,algo_line:prox-slide-avg} as follows
	\begin{align}
	\widetilde{u}_T &= (1 - \theta_T) \widetilde{u}_{T-1} + \theta_T u_T \nonumber \\
	&= \frac{(T-1)(T+2)}{T(T+3)} ((1 - \theta_{T-1}) \widetilde{u}_{T-2} + \theta_{T-1} u_{T-1})  + \frac{2(T+1)}{T(T+3)} u_T \nonumber \\
	&= \frac{(T-2)(T+1)}{T(T+3)} \widetilde{u}_{T-2} + \frac{2(T)}{T(T+3)}  u_{T-1} + \frac{2(T+1)}{T(T+3)}  u_T \nonumber \\
	&\;\; \vdots \nonumber \\
	&= \sum_{t=1}^{T} \frac{2(t+1)}{T(T+3)} u_{t}
	\end{align}
	Finally we get the desired result by setting $\phi = \phi_k$, $\beta = \beta_k$, $T = T_k$, $u_0 = {\tdx}_{k-1}'$, $u = x'$, $u_t = u_{k,t}$, $\widetilde{u}_T = {\tdx}_{k}'$, and ${u}_T = {\dx}_{k}$
	we get the desired inequality
	\begin{align}
	\phi_k({\tdx}_k') - \phi_k (x') &\leq \frac2{T_k(T_k+3)} \frac{\beta_k}2 \|{\dx}_{k-1}'-x'\|^2 - \frac{(T_k+1)(T_k+2)}{T_k(T_k+3)} \frac{\beta_k}2 \|{\dx}_{k}'-x'\|^2 + \nonumber \\ 
	&\;\;\;\;\; \frac{4\,\sum_{t=0}^{T_k-1} (2G + \|\delta_{k,t}\|)^2}{\beta_k T_k(T_k+3) }  + \sum_{t=0}^{T_k-1} \frac{2(t+2)}{T_k(T_k+3)} \Ip{\delta_{k,t}}{x' - u_{k,t}}
	\end{align}
\end{proof}

\subsection{Proof of Lemma~\ref{lem:moreau-properties}}
\label{lem:moreau-properties-pf}
We re-write $f_{\lambda}(x)$ as minimum value of a $\frac1{\lambda}$-strong convex function $\phi_{\lambda, x}$ as follows
\begin{align} \label{eq:moreau-reformula}
f_{\lambda}(x) = \min_{x' \in \bX} \bigg[ \phi_{\lambda, x}(x') \defeq f(x') + \frac{1}{2\lambda} \|x - x'\|^2 \bigg]\,.
\end{align}
Note that $\phi_{\lambda,x}(\cdot)$ is a $(1/\lambda)$-strongly convex function as $f$ is convex and $(1/\lambda) \|\cdot - x\|^2$ is strongly convex, and $f_{\lambda}(x) = \min_{x' \in \bX} \phi_{\lambda, x}(x')$.\\
(a) The existence and uniqueness of $\hat{x}_{\lambda}(x) \in \bX$ follows from the strong convexity of $\phi_{\lambda, x}(\cdot)$ and the fact that $f$ is a proper convex function. Then $f(\hat{x}_{\lambda}(x)) \leq \phi_{\lambda, x}(\hat{x}_{\lambda}(x)) = \min_{x' \in \bX} \phi_{\lambda, x}(x') = f_{\lambda}(x) \leq \phi_{\lambda, x}(x) = f(x)$.
\\ 
(b) Let $g_x \defeq {(x- \hat{x}_\lambda(x))}/\lambda$ for any $x \in \vX$. By $(1/\lambda)$-strong convexity of $\phi_{\lambda,x}(x')$ and $\hat{x}_\lambda(x) = \argmin_{x' \in \bX} \phi_{\lambda,x}(x')$, we have, for any $x' \in \bX$, that 
\begin{align}
\phi_{\lambda,x}(x') &\geq \phi_{\lambda,x}(\hat{x}_\lambda(x)) + \|{x'-\hat{x}_\lambda(x)}\|^2/2\lambda \nonumber \\
\iff f(x') +  \|{x'-x}\|^2/2\lambda &\geq f(\hat{x}_\lambda(x)) + \|{x'-\hat{x}_\lambda(x)}\|^2/2\lambda + \|{x'-\hat{x}_\lambda(x)}\|^2/2\lambda \nonumber \\
\iff f(x') &\geq f(\hat{x}_\lambda(x)) + \Ip{g_x}{x' - \hat{x}_\lambda(x)}
\end{align}
Using this, for any $x, y \in \reals^d$ we get
\begin{align}
f_\lambda(y) - f_\lambda(x) &= f(\hat{x}_\lambda(y)) - f(\hat{x}_\lambda(x)) + (\|\hat{x}_\lambda(y) - y\|^2 - \|\hat{x}_\lambda(x) - x\|^2)/{2\lambda} \nonumber \\
&\geq \Ip{g_x}{\hat{x}_\lambda(y) - \hat{x}_\lambda(x)} + {\lambda/2}(\|g_y\|^2 - \|g_x\|^2) = \Ip{g_x}{y - x} + \lambda/2 \|g_x - g_y\|^2
\end{align}
By instantiating the above for $y\leftarrow x$, $x\leftarrow y$, we also get $f_\lambda(y) - f_\lambda(x) \leq \Ip{g_y}{y - x} - \lambda/2 \|g_x - g_y\|^2$. Combining these two inequalities
\begin{align}
0 \leq \lambda/2 \|g_y-g_x\|^2 \leq f_\lambda(y) - f_\lambda(x) - \Ip{g_x}{y - x} 
&\leq - \lambda/2 \|g_y-g_x\|^2 + \Ip{g_y - g_x}{y - x} \nonumber \\
&\leq - \lambda/2 \|g_y-g_x\|^2 + \|g_y - g_x\|\|y - x\| \nonumber \\
&\leq \|y-x\|^2/2\lambda
\end{align}
This implies that $\lim_{y \to x}  (f_\lambda(y) - f_\lambda(x) - \Ip{g_x}{y - x})/\|y-x\| = 0$. Thus $f_\lambda$ is Frechet differentiable with gradient $\nabla f_\lambda(x) = g_x = (x-\hat{x}_\lambda(x))/\lambda$. The above inequality also implies $f_\lambda$ is convex and $1/\lambda$-smooth.
\\
(c) Let $x \in \bX$. 
Using $1/\lambda$-strong convexity of $\phi_{\lambda,x}$ and $\hat{x}_{\lambda}(x) \in \argmin_{x' \in \bX} \phi_{\lambda,x}(x')$, and $G$-Lipschitzness of $f$ in $\bX$, we get
\begin{align*}
\|x-\hat{x}_{\lambda}(x)\|^2/2\lambda \leq \phi_{\lambda,x}(x) - \phi_{\lambda,x}(\hat{x}_{\lambda}(x))  &= f(x) - f_\lambda(x)\\
&= f(x) - f(\hat{x}_\lambda(x) ) -  \|x-\hat{x}_{\lambda}(x)\|^2/2\lambda \\
&\leq G \|\hat{x}_{\lambda}(x) - x\| - \|x-\hat{x}_{\lambda}(x)\|^2 /2\lambda
\leq G^2 \lambda /2\;. \qed
\end{align*}

\section{Additional details for the experiments in Section \ref{sec:exps}}
\label{sec:exps_details}
For all the experiments we randomly and uniformly sample a point $x_0$ from the surface of the nuclear norm ball of radius $r$. 
For all the figures where we plot the estimated sub-optimality gap: $f(x_k) - \hat{f}^*$,  where $\hat{f}^*$ is the estimated minimum function value calculated by running the \psgd method for a large number of iterations.
We plot the mean (standard error is negligible) of the sub-optimality gap over 10 runs using 10 different initial points $x_0$'s (same 10 initial points for all algorithms).

For experiments in Figures \ref{fig:aproj_imagewoof} and \ref{fig:almo_imagewoof}, we use a subset of the Imagewoof 2.0 dataset~\cite{How2019Imagenette}, which in itself is a subset of the Imagenet dataset \cite{deng2009imagenet}.
The training data, contains $n=400$ samples $\{(A_i, y_i)\}_{i=1}^n$ where $A_i$ is a $224\times224$ grayscale image of one of the two types of dogs (classes n$02087394$ and n$02115641$ in Imagenet dataset) labeled using $y_i \in \{0,1\}$. Note that the effective dimension is $d = 224\times 224 = 50176$). These grayscale images are generated from the raw $8$-bit RGB Imagewoof images using the Pillow python image-processing library~\cite{Cla2020Pillow}, by $(i)$ resizing to $256\times256$ pixels: \verb|resize(256,256)|, $(ii)$ cropping to the central $224\times224$ pixels: \verb|crop(16,16,240,240)|, $(iii)$ converting to the grayscale: \verb|convert(mode=`L')|, and $(iv)$ normalizing by $255.0$ so that the pixel values lie in range $[0,1]$. For incorporating bias scalar into the SVM model we also zero-pad the training images with an additional column and row of zeros to the right and the bottom of the image array $A_i$. We use $r=0.1$ as nuclear norm ball radius of $\cX$, thus $D_\cX = 0.2$. We have access to a deterministic FO.

In Figure~\ref{fig:aproj_imagewoof}, we use a Lipschitz constant of $G=50$. For \aproj~we set $c=40$ and $\varepsilon=5.0$, and for \psgd~we use two stepsize schemes: ($i$) fixed stepsize $D_\cX/(G\sqrt{K})$ with $K=10^{3}$ and ($ii$) diminishing stepsize $D_\cX/(G\sqrt{k})$ with $K=10^3$.

In Figure~\ref{fig:almo_imagewoof}, we use a Lipschitz constant of $G=50$. For \almo~we set $c=40$, $c'=1$ and $\varepsilon=5.0$.
For \fwpsgd~we use two stepsize schemes: ($i$) fixed stepsize $D_\cX/(G\sqrt{K})$ with $K=10^{3}$ and ($ii$) diminishing stepsize $D_\cX/(G\sqrt{k})$ with $K=10^3$. Both of these stepsize schemes use a projection tolerance of $\eta_k G^2/2$. For \randfw we use the standard parameter choices as given in \cite[Theorem 5]{lan2013complexity} with $K = 150$.

In practice, in the deterministic setup with FO, at outer-step $k$, we can use the following criterion for stopping the \proxslide (\Cref{algo_line:proxslide_proj}) procedure early at some $t \geq \widehat{T}_{k-1}$ (defined recursively below with $\hT_0 = 1$) and $t \leq T_k$. Let $\phi_k(x') \defeq f\left(x'\right)  +\Ip{\nabla_{k,x'} }{x'} + \frac{\beta_k}{2}\left\|x'-{\dx}_{k-1}'\right\|^{2}$ and $\widetilde{g}_t \in \partial f(\tu_t)$. Now if 
\begin{align}
\max_{x' \in \cX} \; &\Ip{\tg_t + \nabla_{k,x'}}{\tu_t-x'} - \frac{(t+1)(t+2)}{t(t+3)} {\beta_k} \Ip{u_t - {\dx}_{k-1}'}{x'} \nonumber \\ &\leq \frac{8\,(4G^2 + \sigma^2 )}{\beta_k (T_k+3) } - \frac{\beta_k}{2} \|\tu_t - {\dx}_{k-1}'\|^2 + \frac{(t+1)(t+2)}{t(t+3)} \frac{\beta_k}2 (\|{\dx}_{k-1}'\|^2- \|u_t\|^2) \label{eq:proxslide-stopping-criterion}
\end{align}
then we stop the procedure, set $\widehat{T}_k = t$ and return $(u_t, \tu_t)$. This implies that for $({\dx}_{k}', {\tdx}_{k}') = (u_t, \tu_t)$
\begin{align}
\phi_k({\tu}_t) - \phi_k (x') &\leq \frac2{\widehat{T}_k(\widehat{T}_k+3)} \frac{\beta_k}2 \|{\dx}_{k-1}'-x'\|^2 - \frac{(\widehat{T}_k+1)(\widehat{T}_k+2)}{\widehat{T}_k(\widehat{T}_k+3)} \frac{\beta_k}2 \|u_t-x'\|^2 + \nonumber \\ &\;\;\;\;\; \frac{4\,\sum_{t=0}^{T_k-1} (2G)^2}{\beta_k T_k(T_k+3) }
\end{align}
for all $x' \in \cX$. Now the only change we need to make in the analysis of Theorem \ref{thm:envelope-subgrad-method-thm} is the change of the potential~\eqref{eq:almo-potential-func} to
\begin{align}
\Phi_k \defeq k(k+1) (\Psi_\lambda({\px}_{k},{\px}_{k}') - \Psi_\lambda(x,x')) + \frac{4}{\lambda} (\| {\dx}_{k} - x\|^2 + \frac{(\hT_{k+1}+1)(\hT_{k+1}+2)}{\hT_{k+1}(\hT_{k+1}+3)} \|  {\dx}_{k}' - x'\|^2)
\end{align}

The LHS of \eqref{eq:proxslide-stopping-criterion} is an linear optimization problem whose solution can be easily found as 
\begin{align}
&\lmo{\tg_t + \nabla_{k,x'} + \frac{(t+1)(t+2)}{t(t+3)} {\beta_k} (u_t - {\dx}_{k-1}')} \nonumber \\ = &\lim_{\alpha \to \infty} \po{-\alpha\bigg(\tg_t + \nabla_{k,x'} + \frac{(t+1)(t+2)}{t(t+3)} {\beta_k} (u_t - {\dx}_{k-1}')\bigg)} \,.
\end{align}
We also use a slightly modified $T_{k} = \Big \lceil{ \frac{2G^2\lambda^2 K k^2}{2\tilde{D}}} \Big \rceil$ for our experiments, since the deterministic FO we use, ensures this choice gets the same guarantees as given in our theorems. Also, in our implementation we do not explicitly project ${\dx}_{k}$ onto $\bX$, as in practice this does not seem needed.

In practice, we can eliminate the need for selecting $\varepsilon$ of \aproj by employing $\varepsilon$-doubling trick with warm restarts, which can increase the worse-case iteration complexity by a factor of at most $2$ but oftentimes will accelerate the convergence~\cite[Algorithm 6]{thekumparampil2019efficient}.

 \section{Additional details for applications}
 \label{sec:applications_details}
We refer to \cite{ravi2019deterministic} for some more nonsmooth problems which can be solved using an LMO. In the following subsections, we compare the analytical complexities for solving some of the applications mentioned in Section~\ref{sec:applications}, using different algorithms. 
 
\subsection{$\ell_1$ norm constrained SVM} \label{sec:ell1_SVM}
For simplicity, we work with the vector version of the matrix problems and replace nuclear norm constraint with the $\ell_1$ norm constraint. The standard $\ell_1$ norm constrained soft-margin SVM can be formulated as the the following optimization problem:
\begin{align}
&\min_{x \in \reals^{d}} f(x) = \frac1{n} \sum_{i = 1}^n [f_{i} (x) = \max(0, 1- \Ip{x}{a_{i}} ) ] \nonumber \\
&\text{subject to }\;\;\; \|x\|_{1} \leq \lambda
\end{align}
where $a_i \in \reals^d$ captures the $d$-dimensional feature vector multiplied by a binary class value in $\{-1, 1\}$ and $\cX=\{x \,|\, \|x\|_1 \leq 1 \}$ is the constraint set. We do not include any explicit bias term above, because it can always be incorporated into the model by augmenting $a_{i}$ with a constant dimension. We assume that $n$ is large and therefore we only have access to minibatched stochastic subgradients obtained through minibatching $d$ ($b = \ord(n)$) uniformly sampled (with replacement) training samples. We assume that $f$ is $G_p$-Lipschitz continuous and the variance of any stochastic subgradient is upperbounded by $\sigma_p^2$, both calculated in $\ell_p$ norm $\|\cdot\|_p$, for $p = 1,2$. We define $q \defeq {(1-1/p)^{-1}} \in \{\infty, 2\}$. Then
\begin{align}
G_p &= \max_{\|x\|_1\ \leq \lambda} \|\frac1n \sum_{i = 1}^n \indctr\{ \Ip{x}{a_{i}} < 1\} a_i\|_q \text{, and} \nonumber \\
\sigma_p^2 &= \max_{\|x\|_1\ \leq \lambda} \E_{\{I_j\}_{j=1}^b} \|\frac1n \sum_{i = 1}^n \indctr\{ \Ip{x}{a_{i}} < 1\}  a_i  - \frac1b \sum_{j = 1}^b \indctr\{ \Ip{x}{a_{I_j}} < 1\} a_{I_j}\|^2_q
\end{align}

{\bf PO:} First we study the case of PO (or MO: Mirror descent step oracle) in the high-dimensional ($\poly(G_p, \sigma_p, \lambda, 1/\varepsilon) \ll d$) and large-scale ($1 \ll \poly(n)$) regime. In Table \ref{tab:svm_compare_po} we provide the 
\pocc and \sfocc
of \aproj ($p=1$, Algorithm \ref{algo:envelope_subgrad_method_proj}) and competing nonsmooth methods: \psgd ($p=2$) \cite{goldstein1964convex,levitin1966constrained}, Mirror descent ($p=1$) \cite{nemirovsky1983problem}, Randomized smoothing  ($p=1$ or  $p=2$) \cite{duchi2012randomized}. The $p$ value in brackets marks which $\ell_p$ norm the method uses. By definition $G_1 \leq G_2 \leq \sqrt{d} G_1$ and $\sigma_1 \leq \sigma_2 \leq \sqrt{d} \sigma_1$. Therefore, in this high-dimensional and large-scale regime, and when $G_2 = \ord(\sqrt{d}G_1)$ and $\sigma_2 = \ord(\sqrt{d} \sigma_1)$, \aproj has a more efficient PO-CC than other competing nonsmooth first-order methods, while still maintaining $\Ord(1/\varepsilon^2)$ SFO-CC. Note that PO has a computational complexity of $O(d \log d)$ because it involves sorting~\cite{duchi2008efficient}, MO has a computational complexity of $\Ord(d)$, and SFO has a computational complexity of $\Ord(b(d+n))$ because it involves sampling $b$ vectors from a set of $n$ $d$-dimensional vectors. In practice, sorting could contribute to a significant part of the wall-clock time. 
\begin{table}[]
	\begin{center}
		\label{tab:svm_compare_po}
		{\small
			\begin{tabular}{l c  c}
				\toprule
				{\bf PO based methods (using $\ell_p$ norm)} & & \\ \midrule \midrule
				{\bf Nonsmooth methods ($p=2$)}& {{\bf PO}: $\Ord(d \ln\,d )$} & {{\bf SFO}: $\Ord(d+n)$} \\	\midrule 
				Our \aproj  ($p=2$) [Theorem~\ref{thm:envelope-subgrad-method-cor1}]& $\Ord\big(\frac{G_2}{\varepsilon}\lambda\big)$ & $\Ord\big(\frac{G_2^2+\sigma_2^2}{\varepsilon^2}\lambda^2\big)$ \\ 
				\psgd ($p=2$) & $\Ord\big(\frac{G^2_2}{\varepsilon^2}\lambda^2\big)$ & $\Ord\big(\frac{G_2^2+\sigma_2^2}{\varepsilon^2}\lambda^2\big)$ \\ 
				Randomized smoothing ($p=2$) \cite{duchi2012randomized} & $\Ord\big(d^{1/4}\frac{G_2}{\varepsilon}\lambda\big)$ & $\Ord\big(\frac{G_2^2+\sigma_2^2}{\varepsilon^2}\lambda^2\big)$ \\ 
				\midrule \midrule
				{\bf Nonsmooth methods ($p=1$)}& {{\bf MO}: $\Ord(d)$} & {{\bf SFO}: $\Ord(d+n)$} \\	\midrule 
				Mirror descent ($p=1$) \cite{nemirovsky1983problem} & $\Ord\big(\ln(d+1) \frac{G^2_1}{\varepsilon^2}\lambda^2\big)$ & $\Ord\big(\ln(d+1)\frac{G_1^2+\sigma_1^2}{\varepsilon^2}\lambda^2\big)$ \\ 				
				Randomized smoothing ($p=1$) \cite{duchi2012randomized} & $\Ord\big(\sqrt{d\ln(d+1)} \frac{G_1}{\varepsilon}\lambda\big)$ & $\Ord\big(\ln(d+1) \frac{G_1^2+\sigma_1^2}{\varepsilon^2}\lambda^2\big)$ \\ \midrule \midrule
				{\bf Minimax methods}: $\Ord(n)$ extra memory & {{\bf PO+MO}: $\Ord(d \ln\,d + n)$} & {{\bf SFO}: $\Ord(d+n)$} \\	\midrule 
				Variance reduced Mirror-Prox ($p=1$)\cite{carmon2019variance} & \multicolumn{2}{c}{$\Ord\Big( \frac{dn}{d+n} + \frac{L_{12}}{\varepsilon} \sqrt{\frac{dn}{d+n}} (\lambda \sqrt{n \ln d})\Big)$} \\
				\bottomrule 
			\end{tabular}
		}
	\end{center}
	\caption{Projection: Comparison of PO/MO and SFO calls complexities (PO-CC and SFO-CC) of various methods for $d$-dimensional $\ell_1$ norm constrained SVM with $n$ training samples. SFO uses a batchsize of $b = \ord(n)$. Our \aproj outperforms other nonsmooth methods in PO-CC/MO-CC while still maintaining $\Ord(1/\varepsilon^2)$ SFO-CC. Complexities of methods based on smooth minimax reformulation adversely scale with $n$ or $d$.}
\end{table}

Many nonsmooth convex objectives in machine learning like the hinge loss here can be written as smooth convex-concave minimax objectives of the form
\begin{align}
\min_{x \in \cX} \frac1n \sum_{i=1}^n \max_{y_i \in \cY_i} g_i(x, y_i)
\end{align}
where $g_i(x, \cdot)$ is concave and $g_i$ is $L$-smooth for all $i \in [n]$. However, the iteration/projection complexities of even the best variance reduced algorithms could have a dependence on the number $n$ of additionally introduced dual variables $\{y_i\}_{i=1}^n$~\cite{palaniappan2016stochastic,carmon2019variance}.
Therefore in the regime when comparatively $n$ is large and $\varepsilon$ is moderate ($\poly(1/\varepsilon) \ll n$), it is more efficient to optimize the original stochastic nonsmooth formulation than the smooth minimax reformulation.  

Concretely, the soft-margin SVM problem with a hinge loss, can be reformulated as a saddle point problem of the following form
\begin{align}\label{eq:svm_minimax}
\min_{\|x\|_1 \leq 1} \max_{y \in [0,1]^n}  \frac1n  (y^T\ones - y^T A x)\,.
\end{align}
This smooth saddle point problem is an $\ell_1$-$\ell_2$ matrix game (ignoring possibility of $\ell_\infty$ optimization due to limited literature) which is $L_{12}$-smooth, where
\begin{align}
L_{12} &= \max_{\|x\|_1\ \leq 1} \max_{\|y\|_2 \leq 1} \frac1n y^T A x =\frac1n \max_{i=1,\ldots,n} \|a_i\|_2 \,.
\end{align}
Note that the primal (in $\ell_1$ norm) and dual (in $\ell_2$ norm) space diameters are $D_\cX = \Ord(\lambda)$ and $D_\cY = \Ord(\sqrt{n})$ respectively. 

Next we derive the MO and SFO calls complexities (MO-CC and SFO-CC) of the variance reduced Mirror-prox method \cite{carmon2019variance}. For any stepsize $\alpha \leq \frac\varepsilon{D_\cX D_\cY \sqrt{\ln d}}$, this algorithm runs for $K = \Ord(\frac{\alpha D_\cX D_\cY \sqrt{\ln d}}{\varepsilon})$ outer iterations, each of which uses $T = 1+ \frac{L_{12}^2}{\alpha^2}$ SFO calls and one FO call, and $T+1$ primal and dual MO calls. Computational complexity of 
\begin{itemize}
\item Primal MO is $O(d \log d)$ since it involves sorting,
\item Dual MO is $O(n)$ since it involves normalization of each dual dimension,
\item FO is $\Ord(dn)$ since it involves $d \times n$-matrix vector products, and,
\item SFO is $\Ord(d+n)$ because it involves sampling from two set of $n$ and $d$ ($d$ and $n$-dimensional, respectively) vectors. 
\end{itemize}
We assume that the algorithm uses $\widetilde{T} = 1 + \frac{L_{12}^2}{\alpha^2} + \frac{dn}{d+n} = \Ord(\frac{L_{12}^2}{\alpha^2} + \frac{dn}{d+n})$ SFO calls per outer iteration, because computationally it is equivalent to $T = 1 + \frac{L_{12}^2}{\alpha^2}$ SFO calls and one FO call per outer iteration. Using the suggested stepsize $\alpha = \max(\frac\varepsilon{D_\cX D_\cY \sqrt{\ln d}}, L_{12} \sqrt{\frac{d+n}{dn}})$, we get that 
\begin{align}
[\text{MO-CC} = \Ord(K \cdot T)] = \Ord\Big( \frac{dn}{d+n} + \frac{L_{12}}{\varepsilon} \sqrt{\frac{dn}{d+n}} (\lambda \sqrt{n \ln d})\Big) =  [K \cdot \widetilde{T} = \text{SFO-CC}] \,.
\end{align}
In very high dimensional regime ($n \ll d$) or very large-scale regime ($d \ll n$), MO-CC of this smooth minimax formulation is $\Ord(d)$ or $\Ord(n)$ larger than PO-CC for \aproj. Further more the former method uses extra $\Theta(n)$ extra space for storing the dual variables.

{\bf LMO:} Next we study the case of LMO in the high-dimensional ($\poly(G_p, \sigma_p, \lambda, 1/\varepsilon) \ll d$) and large-scale ($1 \ll \poly(n)$) regime. In Table \ref{tab:svm_compare_lmo} we provide the LMO and SFO calls complexities of \almo ($p=1$, Algorithm \ref{algo:envelope_subgrad_method_proj}) and competing nonsmooth methods: \fwpsgd ($p=2$)---projection approximated with Frank-Wolfe method (Appendix~\ref{sec:fw_proj_subgrad_method}), and Randomized Frank-Wolfe method ($p=1$ or  $p=2$) \cite{lan2013complexity}. The $p$ value in brackets marks which $\ell_p$ norm the method uses. By definition $G_1 \leq G_2 \leq \sqrt{d} G_1$ and $\sigma_1 \leq \sigma_2 \leq \sqrt{d} \sigma_1$. Therefore, in this high-dimensional and large-scale regime, \almo has a more efficient dimension-free LMO-CC $\Ord(G_2^2 \lambda^2/\varepsilon^2)$ than other competing nonsmooth first-order methods, while still maintaining optimal $\Ord(1/\varepsilon^2)$ SFO-CC. Note that LMO has a computational complexity of $O(d)$ because it uses just one pass over a $d$-dimensional vector.

A competing method based on the smooth minimax reformulation is SP+VR-MP which combines ideas from Semi-Proximal~\cite{he2015semi} and Variance reduced~\cite{carmon2019variance} Mirror-Prox methods.
Here SP+VR-MP uses the variance reduced Mirror-prox method \cite{carmon2019variance} in the $\ell_2$-$\ell_2$ setting to optimize \eqref{eq:svm_minimax} and then approximates the projection steps with Frank-Wolfe (FW) method. This is an $L_{22}$-smooth minimax problem with
\begin{align}
L_{22} &= \max_{\|x\|_2\ \leq 1} \max_{\|y\|_2 \leq 1} \frac1n y^T A x =\frac1n \|A\|_2 \,.
\end{align}
where $\|A\|_2$ is the spectral norm of the matrix $A$, but the algorithm we are discussing will depend on
\begin{align}
\widetilde{L}_{22} &= \frac1n \|A\|_F \,.
\end{align}
where $\|A\|_F$ is the Frobenius norm of the matrix $A$. Note that $\frac1{\sqrt{\min(n, d)}} \|A\|_F \leq \|A\|_2 \leq \|A\|_F$. The primal and dual space diameters are again $D_\cX = \Ord(\lambda)$ and $D_\cY = \Ord(\sqrt{n})$, respectively. 

For each of the projection steps, the Frank-Wolfe method solves an $(\alpha+ 10\frac{L_{22}^2}\alpha)$-smooth convex optimization problem up to an error $\Ord(\varepsilon)$. Therefore each of these uses at most $\widehat{T} =  \lceil (\alpha+ 10\frac{L_{22}^2}\alpha) \lambda^2/\varepsilon \rceil$ LMO calls. Thus using similar arguments as the PO setting and using the suggested stepsize $\alpha = \max(\frac\varepsilon{D_\cX D_\cY}, \widetilde{L}_{22} \sqrt{\frac{d+n}{dn}})$, $K = \Ord(\frac{\alpha D_\cX D_\cY}{\varepsilon})$ outer iterations, $\widetilde{T} = 1 + \frac{\widetilde{L}_{22}^2}{\alpha^2} = \Ord(\frac{\widetilde{L}_{22}^2}{\alpha^2})$ SFO calls per outer iteration, and $\widetilde{T} = 1 + \frac{\widetilde{L}_{22}^2}{\alpha^2} + \frac{dn}{d+n} = \Ord(\frac{\widetilde{L}_{22}^2}{\alpha^2} + \frac{dn}{d+n})$ effective number of SFO calls per outer iteration, we get that 
\begin{align}
[\text{SFO-CC}  = K \cdot \widetilde{T}]  
&= \Ord\Big( \frac{dn}{d+n} + \frac{\widetilde{L}_{22}}{\varepsilon} \sqrt{\frac{dn}{d+n}} (\lambda \sqrt{n})\Big) \,,
\end{align}
and
\begin{align}
&\;\;\; [\text{LMO-CC} = \Ord(K \cdot {T}) \cdot \widehat{T} ] \nonumber \\
&= \Ord\Big( \Big[\frac{dn}{d+n} + \frac{\widetilde{L}_{22}}{\varepsilon} \sqrt{\frac{dn}{d+n}} (\lambda \sqrt{n})\Big] \cdot \Big[ 1 + (\alpha+ \frac{L_{22}^2}\alpha) \frac{ D_{\cX}^2}{\varepsilon}\Big]\Big) \nonumber \\
&= \Ord\Big( \Big[\frac{dn}{d+n} + \frac{\widetilde{L}_{22}}{\varepsilon} \sqrt{\frac{dn}{d+n}} (\lambda \sqrt{n})\Big] \cdot \Big[ 1 + \frac\lambda{\sqrt{n}} + \frac{\widetilde{L}_{22} \, \lambda^2}{\varepsilon} \sqrt{\frac{dn}{d+n}} \Big]\Big) \nonumber \\
&= \text{SFO-CC} + \Ord\Big( \frac{d\sqrt{n}\, \lambda}{d+n} + \frac{\widetilde{L}_{22} \, \lambda^2}{\varepsilon} \Big(\frac{dn}{d+n}\Big)^{\frac32} + \frac{\widetilde{L}_{22}^2 \, \lambda^3 \sqrt{n}}{\varepsilon^2} \Big(\frac{dn}{d+n}\Big) \Big) 
\end{align}
In very high dimensional regime ($n \ll d$) or very large-scale regime ($d \ll n$), LMO-CC of this smooth minimax reformulation is $\Ord(d)$ or $\Ord(n)$ larger than LMO-CC for \almo. Further more the former method uses extra $\Theta(n)$ extra space for storing the dual variables.

\begin{table}[]
	\begin{center}
		\label{tab:svm_compare_lmo}
		{\small
			\begin{tabular}{l c  c}
				\toprule
				{\bf LMO based methods (using $\ell_p$ norm)} & & \\ \midrule \midrule
				{\bf Nonsmooth methods ($p=2$)}& {{\bf LMO}: $\Ord(d)$} & {{\bf SFO}: $\Ord(d+n)$} \\	\midrule 
				\almo ($p=2$) [Theorem~\ref{thm:envelope-subgrad-method-cor2}] & $\Ord\big(\frac{G_2^2}{\varepsilon^2}\lambda^2\big)$ & $\Ord\big(\frac{G_2^2+\sigma_2^2}{\varepsilon^2}\lambda^2\big)$ \\ 
				\fwpsgd ($p=2$) [Theorem~\ref{thm:fw_proj_subgrad_method-thm}] & $\Ord\big(\frac{G_2^4+\sigma_2^4}{\varepsilon^4}\lambda^4\big)$ & $\Ord\big(\frac{G_2^2+\sigma_2^2}{\varepsilon^2}\lambda^2\big)$ \\ 
				Rand.~Frank-Wolfe ($p=2$) \cite{lan2013complexity} & $\Ord\big(d^{1/2}\frac{G_2^2}{\varepsilon^2}\lambda^2\big)$ & $\Ord\big( \frac{G_2^4+\sigma_2^4}{\varepsilon^4}\lambda^4\big)$ \\ 
				\midrule \midrule
				{\bf Nonsmooth methods ($p=1$)}& {{\bf LMO}: $\Ord(d)$} & {{\bf SFO}: $\Ord(d+n)$} \\	\midrule 
				Rand.~Frank-Wolfe ($p=1$) \cite{lan2013complexity} & $\Ord\big(d\ln(d+1) \frac{G_1^2}{\varepsilon^2}\lambda^2\big)$ & $\Ord\big(\ln^2(d+1) \frac{G_1^4+\sigma_1^4}{\varepsilon^4}\lambda^2\big)$ \\ 
				\midrule \midrule
				{\bf Minimax methods}: $\Ord(n)$ extra memory &{{\bf LMO}: $\Ord(d)$} & {{\bf SFO}: $\Ord(d+n)$} \\	\midrule 
				SP~\cite{he2015semi}+VR~\cite{carmon2019variance}-MP ($p=2$) & 
				{\bf SFO-CC} $+ \; \Ord\Big( \frac{d\sqrt{n}\, \lambda}{d+n} +$ &
				$\Ord\Big( \frac{dn}{d+n} +$ \\
				& 
				$\frac{\widetilde{L}_{22} \, \lambda^2}{\varepsilon} \Big(\frac{dn}{d+n}\Big)^{\frac32} + \frac{\widetilde{L}_{22}^2 \, \lambda^3 \sqrt{n}}{\varepsilon^2} \Big(\frac{dn}{d+n}\Big) \Big)$ &
				$+ \; \frac{\widetilde{L}_{22}}{\varepsilon} \sqrt{\frac{dn}{d+n}} (\lambda \sqrt{n})\Big)$
				\\
				\bottomrule 
			\end{tabular}
		}
	\end{center}
	\caption{Linear minimization oracle: LMO and SFO calls complexity (LMO-CC and SFO-CC) of various methods for $d$-dimensional $\ell_1$ norm constrained SVM with $n$ training samples. SFO uses a batchsize of $b = \ord(n)$. SP+VR-MP combines ideas from Semi-Proximal~\cite{he2015semi} and Variance reduced~\cite{carmon2019variance} Mirror-Prox methods. Our \almo outperforms other nonsmooth methods in LMO-CC while still maintaining $\Ord(1/\varepsilon^2)$ SFO-CC. Complexities of method based on smooth minimax reformulation adversely scale with $n$ or $d$.}
\end{table}

Similar arguments hold for the nuclear norm constrained Matrix SVM~\cite{wolf2007modeling}, so that \aproj/\almo outperforms other nonsmooth methods in some regime, where $n$ or $d$ is large and $\varepsilon$ is relatively moderate, and complexities of smooth minimax reformulation based methods scales adversely with $d$ and $n$. For this case, the gain in the actual wall-clock time would be even more stark than vector SVM due to the computation of SVD/largest eigenvalue, which is required for implementing PO/LMO.
 
 \subsection{SVM with hard constraints~\cite{nguyen2014efficient}} \label{sec:hard_SVM} 
Soft-margin SVM could be provided with some hard constraints,  so that the classifier is forced to always predict the correct labels for a subset (of size $k$) of important ``gold'' training examples. This problem can be formulated as a nonsmooth constrained optimization problem with a large number of linear constraints, as follows 
 \begin{align}
\min_{x \in \reals^{d}} \;\;\;\; &\frac1n \sum_{i = 1}^n \max(0, 1- \Ip{x}{a_{i}}) \nonumber \\
\text{subject to, }\;\;\; &1 \leq \Ip{x}{\widetilde{a}_{j}} \,, \; \forall j =1,\ldots,k \nonumber \\
 &\|x\|_1 \leq \lambda
 \end{align}
 We can solve this nonsmooth convex problem using first-order methods using projection onto hard constraints set: $\cX = \{x \,|\, \|x\|_1 \leq \lambda \text{ and } 1 \leq \Ip{x}{\widetilde{a}_{j}}\,, \forall j =1,\ldots,k\}$. This projection can be can be implemented using linear programming methods, however it is computationally costly. Therefore PO-CC efficiency is critical, and just as in the case of SVM with $\ell_1$ norm constraint (Section \ref{sec:ell1_SVM}), our \aproj method achieves smallest $\Ord(1/\varepsilon)$ dimension-free PO-CC which is better than other competing methods. \pocc and \sfocc are the same as given in Table~\ref{tab:svm_compare_po}.
 
Note that, first-order methods using one projection \cite{mahdavi2012stochastic} cannot be applied here, since they need the constraint set to be written in the functional form: $c(x) \leq 0$, such that $\rho \leq \| g \|$ for all $g \in \partial c(x)$, for some $\rho > 0$. This is not true for general set of linear constraints, where a pathological case can occur when two linear constraints have almost identical normal vectors.
\end{document}